\documentclass[oneside,english]{amsart}
\usepackage[T1]{fontenc}
\usepackage[latin9]{inputenc}
\usepackage{xcolor}
\usepackage{mathrsfs}
\usepackage{amstext}
\usepackage{amsthm}

\makeatletter
\numberwithin{equation}{section}
\numberwithin{figure}{section}
\theoremstyle{plain}
\newtheorem{thm}{\protect\theoremname}
  \theoremstyle{plain}
  \newtheorem{cor}[thm]{\protect\corollaryname}
  \theoremstyle{remark}
  \newtheorem{rem}[thm]{\protect\remarkname}
  \theoremstyle{plain}
  \newtheorem{lem}[thm]{\protect\lemmaname}

\RequirePackage{fix-cm}

\makeatother

\usepackage{babel}
  \providecommand{\corollaryname}{Corollary}
  \providecommand{\lemmaname}{Lemma}
  \providecommand{\remarkname}{Remark}
\providecommand{\theoremname}{Theorem}

\begin{document}

\title[The fundamental Lepage form in variational theory]{The fundamental Lepage form \\
in variational theory for submanifolds}

\author{Zbyn\v{e}k Urban and J\'an Brajer\v{c}\'ik}

\address{Z. Urban, Department of Mathematics and Descriptive Geometry\\
V\v{S}B-Technical University of Ostrava\\
17. listopadu 15, 708 33 Ostrava-Poruba, Czech Republic}

\email{zbynek.urban@vsb.cz; urbanzp@gmail.com}

\address{J. Brajer\v{c}\'ik, Department of Physics, Mathematics and Techniques\\
University of Pre\v{s}ov, 17. novembra 1, 081 16 Pre\v{s}ov, Slovakia}

\email{jan.brajercik@unipo.sk}
\begin{abstract}
A setting for global variational geometry on Grassmann fibrations
is presented. The integral variational functionals for finite dimensional
immersed submanifolds are studied by means of the fundamental Lepage
equivalent of a homogeneous Lagrangian, which can be regarded as a
generalization of the well-known Hilbert form in the classical mechanics.
Prolongations of immersions, diffeomorphisms and vector fields to
the Grassmann fibrations are introduced as geometric tools for the
variations of immersions. The first infinitesimal variation formula
together with its consequences, the Euler\textendash Lagrange equations
for extremal submanifolds and the Noether theorem for invariant variational
functionals are proved. The theory is illustrated on the variational
functional for minimal submanifolds.
\end{abstract}



\thanks{The authors are very grateful to Professors Demeter Krupka and David
J. Saunders for research discussions during the 22nd International
Summer School on Global Analysis and Applications, Krakow, August
2017. ZU also acknowledges the SAIA scholarship programme, and hospitality
of the Department of Physics, Mathematics and Techniques, University
of Pre\v{s}ov.}
\maketitle

\section{Introduction\label{sec:1}}

The subject of this paper is the theory of variational functionals
for \textit{\textcolor{black}{submanifolds}}, where extremals become
submanifolds of Euclidean spaces or general finite dimensional smooth
manifolds. Thus, the variational variables will be rather sets than
mappings between manifolds. This theory requires adequate underlying
geometric structures, the quotient topological spaces the higher-order
\textit{\textcolor{black}{Grassmann fibrations}}\textcolor{black}{.
Following the pioneering work by Dedecker \cite{Dedecker}, the theory
has been studied by different authors (cf. Crampin and Saunders \cite{CramSaun1},
Manno and Vitolo \cite{MannoVitolo}, Grigore \cite{Grigore-Handbook}),
who employed }\textit{\textcolor{black}{different}}\textcolor{black}{{}
geometric structures and variational objects defined on them. The
variational theory for }\textit{\textcolor{black}{one}}\textcolor{black}{-dimensional
immersed submanifolds (}\textit{\textcolor{black}{fibered mechanics}}\textcolor{black}{),
including the Noether-type invariance theory, was also studied by
Urban and Krupka \cite{UK-Acta,UK-IJGMMP,UK-AMAPN} with a direct
use of Grassmann fibrations as the underlying spaces.}

\textcolor{black}{Our main aim is to develop foundations of multiple-integral
variational functionals for submanifolds, defined by differential
forms with specific properties \textendash{}} \textit{\textcolor{black}{Lepage
forms}}. Roughly speaking, Lepage forms represent a far-going generalization
of the well-known Cartan form from the calculus of variations of simple
integral problems and classical mechanics. Replacing the initial Lagrangian
by its Lepage equivalent, one obtains the same variational functional,
but additionaly the geometric and variational properties of the functional
are described by geometric operations acting on the corresponding
Lepage equivalent. For a review of basic properties and results of
the theory of Lepage forms in the calculus of variations, see Krupka,
Krupkov\'a, and Saunders \cite{KKS}.

We introduce the concept of a Lepage form for differential forms on
manifolds of regular velocities (jets of immersions) with the help
of the canonical embedding into fibered velocity spaces. It appears
that these forms can be canonically projected onto the Grassmann fibrations,
hence we obtain crucial objects for studying variational properties
(namely variations, extremals, and invariance properties) of the integral
functionals. To this purpose we utilize the concept of the \textit{\textcolor{black}{fundamental
Lepage equivalent}} of a Lagrangian, introduced by Krupka \cite{Krupka-Fund.Lep.eq.}
on \textit{\textcolor{black}{first-order}} jet prolongation of a~fibered
manifold over an $n$-dimensional base. This particular Lepage equivalent
is also characterized by the following important property: it is closed
if and only if the corresponding Lagrangian is \textit{\textcolor{black}{trivial}}
(i.e. the Euler\textendash Lagrange expressions vanish identically).
This fact can also be profitably applied in formulation of first-order
\textit{\textcolor{black}{local}} variational principles (cf. Brajer\v{c}\'ik
and Krupka \cite{BrajercikKrupka}). In such situation, a global Lagrangian
for the variational functional need not be defined. The problem how
to reconstruct the variational functional from the local data occurs
in many physical theories (see Giachetta, Mangiarotti, and Sardanashvily
\cite{GMS-FieldTheory}).

In Section 2, we briefly summarize basic concepts of Lepage forms
in first-order variational field theory on fibered manifolds, including
a description of the well-known examples of Lepage equivalents, namely
the Poincar\'e\textendash Cartan form, the fundamental Lepage form,
and the Carath\'eodory form (cf. Carath\'eodory \cite{Caratheodory}).
Section 3 contains the geometric structure of manifolds of velocities
and Grassmann fibrations, adapted to our setting. A particular attention
is devoted to the Grassmann prolongations of diffeomorphisms and vector
fields, used later on within the calculus of variations. In Section
4, we study the fundamental Lepage equivalent of a~positive-homogeneous
Lagrangian, in particular, we derive its local structure and prove
that this differential $n$-form is defined on the Grassmann fibration.
Necessary and sufficient conditions for a function on manifold of
regular velocities to be positive homogeneous (the Zermelo conditions,
see e.g. McKiernan \cite{McKiernan}, Urban and Krupka \cite{UK-Debrecen})
are applied. In this sense, we follow the idea of Krupka \cite{Krupka-Debrecen},
who studied the unique Lepage equivalent of an $r$-th order Lagrangian
in fibered mechanics (the generalization of the Cartan form) under
assumption of a positive-homogeneous Lagrangian hence obtaining a
generalization of the Hilbert form. An alternative approach was applied
by Crampin and Saunders \cite{CramSaun1}, whose starting object is
the Carath\'eodory form in first-order field theory for a positive-homogeneous
Lagrangian, resulting into the Hilbert\textendash Carath\'eodory
form, which\textit{\textcolor{black}{{} differs}} from the fundamental
Lepage equivalent (see Remark \ref{Remark:Hilbert-Caratheodory}).
These two Lepage equivalents are, however, very closely related as
they both define the minimal submanifold problems (Section \ref{sec:6}).

In Section 5, the first-order variational field theory for submanifolds
is developed. First we study conditions under which a differential
$n$-form on the manifold of regular $n$-velocities is a Lepage form
(Theorem \ref{Theorem:LepCond}), and observe that horizontal component
of a Lepage form is given by a positive-homogeneous function. This
allows us to employ the fundamental Lepage equivalent of a homogeneous
Lagrangian as a basic element of the theory. We derive the infinitesimal
first variation formula and its consequences for extremals and conservation
laws in a global sense. Furthermore, we extend the classical invariant
variational principles and the Noether theory (see e.g. Kossmann-Schwarzbach
\cite{Kossmann}, Krupka \cite{Krupka-Book}) to functionals given
by Lepage forms on the Grassmann fibrations.

The results and methods of this work are illustrated on classical
example of the variational functional for minimal submanifolds (Section
\ref{sec:6}). In particular, we show that the fundamental Lepage
equivalent and the Hilbert\textendash Carath\'eodory equivalent of
the minimal submanifold Lagrangian coincide. Then we analyze the invariance
properties of the variational problem of minimal surfaces ($n=2$).
It turns out in this example that the \textquotedbl{}conservation
law\textquotedbl{} equations are completely equivalent with the Euler\textendash Lagrange
equations for extremals.

Basic underlying structures, well adapted to this paper, can be found
in Grigore and Krupka \cite{Grigore}. Throughout, we use the standard
geometric concepts: the exterior derivative $d$, the contraction
$i_{\Xi}\rho$ and the Lie derivative $\partial_{\Xi}\rho$ of a differential
form $\rho$ with respect to a vector field $\Xi$, and the pull-back
operation $*$ acting on differential forms. If $(U,\varphi)$, $\varphi=(x^{j})$,
is a chart on smooth manifold $X$, we set\textcolor{brown}{
\begin{align*}
{\color{black}\omega_{0}} & {\color{black}=\frac{1}{n!}\varepsilon_{j_{1}j_{2}\ldots j_{n}}dx^{j_{1}}\wedge dx^{j_{2}}\wedge\ldots\wedge dx^{j_{n}},}\\
{\color{black}\omega_{j}} & {\color{black}=i_{\partial/\partial x^{j}}\omega_{0}=\frac{1}{(n-1)!}\varepsilon_{ji_{2}\ldots i_{n}}dx^{i_{2}}\wedge\ldots\wedge dx^{i_{n}},}
\end{align*}
}where $\varepsilon_{i_{1}i_{2}\ldots i_{n}}$ is the Levi-Civita
permutation symbol.

\section{Lepage forms in first-order field theory on fibered spaces \label{sec:2}}

In this section we summarize basic ideas and results of \textit{first-order}
Lepage forms in the global calculus of variations on fibered spaces
over $n$-dimensional manifolds. A~comprehensive higher-order treatment
can be found in Krupka \cite{Krupka-Book}, and in particular first-order
field theory was studied by Voln\'a and Urban \cite{Volna}.

Throughout, we denote by $Y$ a fibered manifold of dimension $n+M$
over an $n$-dimensional base $X$ with projection $\pi:Y\rightarrow X$
the surjective submersion. $J^{1}Y$ (resp. $J^{2}Y$) denotes the
first (resp. second) jet prolongation of $Y$ whose elements are jets
$J_{x}^{1}\gamma$ (resp. $J_{x}^{2}\gamma$) of sections $\gamma$
of $\pi$ with source $x\in X$ and target $\gamma(x)\in Y$. The
canonical jet projections $\pi^{1}:J^{1}Y\rightarrow X$ and $\pi^{1,0}:J^{1}Y\rightarrow Y$
(resp. $\pi^{2}:J^{2}Y\rightarrow X$, $\pi^{2,0}:J^{2}Y\rightarrow Y$,
$\pi^{2,1}:J^{2}Y\rightarrow J^{1}Y$), are defined by $\pi^{1}(J_{x}^{1}\gamma)=x$,
$\pi^{1,0}(J_{x}^{1}\gamma)=\gamma(x)$ (resp. $\pi^{2}(J_{x}^{2}\gamma)=x$,
$\pi^{2,0}(J_{x}^{2}\gamma)=\gamma(x)$, $\pi^{2,1}(J_{x}^{2}\gamma)=J_{x}^{1}\gamma$).
The jet prolongation $J^{1}\gamma$ of a section $\gamma$ (resp.
$J^{2}\gamma$), defined on an open subset of $X$, is given by $J^{1}\gamma(x)=J_{x}^{1}\gamma$
(resp. $J^{2}\gamma(x)=J_{x}^{2}\gamma$). For an open subset $W$
of $Y$ we put $W^{1}=(\pi^{1,0})^{-1}(W)$, $W^{2}=(\pi^{2,0})^{-1}(W)$.
Let $(V,\psi)$, $\psi=(x^{j},y^{K})$, be a fibered chart on $Y$,
and denote by $(V^{1},\psi^{1})$, $\psi^{1}=(x^{j},y^{K},y_{l}^{K})$
(resp. $(V^{2},\psi^{2})$, $\psi^{2}=(x^{j},y^{K},y_{l}^{K},y_{lk}^{K})$),
the associated fibered chart on $J^{1}Y$ (resp. $J^{2}Y$), and by
$(\pi(V),\varphi)$, $\varphi=(x^{j})$, the associated chart on $X$,
where $y_{l}^{K}(J_{x}^{1}\gamma)=D_{l}(y^{K}\gamma\varphi^{-1})(\varphi(x))$,
$y_{lk}^{K}(J_{x}^{2}\gamma)=D_{l}D_{k}(y^{K}\gamma\varphi^{-1})(\varphi(x))$,
and $1\leq j\leq n$, $1\leq K\leq M$, $1\leq l\leq k\leq n$. A
tangent vector $\xi$ at a point $y\in Y$ is said to be $\pi$-vertical,
if $T\pi\cdot\xi=0$, and a differential form $\rho$ on $Y$ is said
to be $\pi$-horizontal, if for every point $y\in Y$ the contraction
$i_{\xi}\rho(y)$ vanishes whenever $\xi\in T_{y}Y$ is a $\pi$-vertical
vector. The concepts of $\pi^{1}$-, $\pi^{2}$-, $\pi^{1,0}$-, and
$\pi^{2,0}$-horizontal forms are introduced analogously. A vector
field $\Xi$ on $Y$ is called $\pi$\textit{\textcolor{black}{-projectable}},
if there exists a~vector field $\xi$ on $X$ such that $T\pi\cdot\Xi=\xi\circ\pi$.
In a~fibered chart $(V,\psi)$, $\psi=(x^{j},y^{K})$, a~$\pi$-projectable
vector field $\Xi$ has an expression $\Xi=\xi^{j}(\partial/\partial x^{j})+\Xi^{K}(\partial/\partial y^{K})$,
where $\xi^{j}=\xi^{j}(x^{l})$, $\Xi^{K}=\Xi^{K}(x^{l},y_{k}^{L})$.

Let $q\geq1$ be an integer. For any open set $W\subset Y$, we denote
by $\Omega_{q}^{1}W$ (resp. $\Omega_{q}^{2}W$) the $\Omega_{0}^{1}W$-module
(resp. $\Omega_{0}^{2}W$-module) of smooth differential $q$-forms
defined on $W^{1}$ (resp. $W^{2}$), where $\Omega_{0}^{1}W$ (resp.
$\Omega_{0}^{2}W$) is the ring of smooth functions on $W^{1}$ (resp.
$W^{2}$). Clearly, $\pi^{1}$-horizontal (resp. $\pi^{1,0}$-horizontal)
$q$-forms on $W^{1}$ constitute submodule of the $\Omega_{0}^{1}W$-module
$\Omega_{q}^{1}W$, denoted by $\Omega_{q,X}^{1}W$ (resp. $\Omega_{q,Y}^{1}W$);
the modules of $\pi^{2}$-horizontal (resp. $\pi^{2,0}$-horizontal)
$q$-forms on $W^{2}$ are denoted by $\Omega_{q,X}^{2}W$ (resp.
$\Omega_{q,Y}^{2}W$). A morphism of exterior algebras $\Omega_{q}^{1}W\ni\rho\rightarrow h\rho\in\Omega_{q,X}^{2}W$,
defined with respect to any fibered chart $(V,\psi)$, $\psi=(x^{i},y^{K})$,
by the identities,

\[
hf=f\circ\pi^{2,1},\quad hdx^{i}=dx^{i},\quad hdy^{K}=y_{k}^{K}dx^{k},\quad hdy_{j}^{K}=y_{jk}^{K}dx^{k},
\]
where $f:V^{1}\rightarrow\mathrm{\mathbf{R}}$ is a differentiable
function, is called the $\pi$-\textit{horizontalization}. In particular,
$hdf=(d_{i}f)dx^{i}$, where $d_{i}f=\partial f/\partial x^{i}+(\partial f/\partial y^{K})y_{i}^{K}+(\partial f/\partial y_{j}^{K})y_{ji}^{K}$
is the $i$-th formal derivative operator associated to $(V,\psi)$.
Note that for any section $\gamma$ of $Y$, we have $J^{1}\gamma^{*}\rho=J^{2}\gamma^{*}h\rho$.
A differential $q$-form $\rho\in\Omega_{q}^{1}W$ is said to be \textit{contact},
if $J^{1}\gamma^{*}\rho=0$ for all sections $\gamma$ of $Y$ defined
on an open subset of $X$ with values in $W$; this condition is equivalent
to $h\rho=0$. If $(V,\psi)$, $\psi=(x^{i},y^{K})$, is a~fibered
chart on $Y$, then the forms $dx^{i}$, $\omega^{K}$, $dy_{j}^{K}$,
where
\begin{equation}
\omega^{K}=dy^{K}-y_{l}^{K}dx^{l},\label{eq:omega}
\end{equation}
constitute a basis of linear forms on $V^{1}$. For $2\leq q\leq n$,
any contact $q$-form on $W^{1}$ is locally generated by the contact
forms $\omega^{K}$ and $d\omega^{K}$ (cf. Krupka \cite{Krupka-Book}).
Any differential $q$-form $\rho\in\Omega_{q}^{1}W$ has a unique
invariant decomposition
\[
(\pi^{2,1})^{*}\rho=h\rho+\sum_{k=1}^{q}p_{k}\rho,
\]
where $p_{k}\rho$ is the $k$-\textit{contact component} of $\rho$,
containing exactly $k$ exterior factors $\omega^{K}$ in any fibered
chart $(V,\psi)$. 

Any element $\lambda\in\Omega_{n,X}^{1}W$, i.e. a $\pi^{1}$-horizontal
$n$-form on the open set $W^{1}\subset J^{1}Y$, is called a \textit{Lagrangian}
for $Y$ of order $1$. In a fibered chart $(V,\psi)$, $\psi=(x^{i},y^{K})$,
where $V\subset W$, $\lambda$ is expressed by
\[
\lambda=\mathscr{L}\omega_{0},
\]
where $\omega_{0}=dx^{1}\wedge dx^{2}\wedge\ldots\wedge dx^{n}$ is
the (local) volume element and $\mathscr{L}:V^{1}\rightarrow\mathbf{R}$
is the \textit{Lagrange function}, associated to $\lambda$.

In accordance with the general theory of Lepage forms (cf. Krupka
\cite{Krupka-Lepage}), we say that an $n$-form $\Theta_{\lambda}\in\Omega_{n}^{1}W$
is a\textit{~Lepage equivalent} of $\lambda$ on $W^{1}\subset J^{1}Y$,
if the following two conditions are satisfied: 

(a) $h\Theta_{\lambda}=(\pi^{2,1})^{*}\lambda$ (i.e. $\Theta_{\lambda}$
is \textit{equivalent }with $\lambda$), and 

(b) $hi_{\xi}d\Theta_{\lambda}=0$ for arbitrary $\pi^{1,0}$-vertical
vector field $\xi$ on $W^{1}$ (i.e. $\Theta_{\lambda}$ is a\textit{~Lepage
form}). 

The following two theorems describe the structure of Lepage forms
and Lepage equivalents of a Lagrangian (see Krupka \cite{Krupka-Lepage,Krupka-Book}).
\begin{thm}
\label{Theorem:LepageForms}An $n$-form $\rho\in\Omega_{n}^{1}W$
is a Lepage form if and only if for any fibered chart $(V,\psi)$,
$\psi=(x^{i},y^{K})$, on $Y$, where $V\subset W$,
\begin{equation}
(\pi^{2,1})^{*}\rho=\Theta+d\mu+\eta,\label{eq:LepFormExp}
\end{equation}
where the principal component $\Theta$ is expressed as
\[
\Theta=f_{0}\omega_{0}+\left(\frac{\partial f_{0}}{\partial y_{j}^{K}}-d_{p}\frac{\partial f_{0}}{\partial y_{pj}^{K}}\right)\omega^{K}\wedge\omega_{j}+\frac{\partial f_{0}}{\partial y_{ij}^{K}}\omega_{i}^{K}\wedge\omega_{j},
\]
$f_{0}$ is a differentiable function on $V^{2}\subset J^{2}Y$, given
by $h\rho=f_{0}\omega_{0}$, $\mu$ is a contact $(n-1)$-form, and
an $n$-form $\eta$ has the order of contactness $\geq2$.
\end{thm}

\begin{cor}
Let $\lambda\in\Omega_{n}^{1}W$ be a Lagrangian of order $1$, expressed
by $\lambda=\mathscr{L}\omega_{0}$. A Lepage form $\rho\in\Omega_{n}^{1}W$
is a Lepage equivalent of $\lambda$ if and only if the principal
component $\Theta$ of $\rho$ is defined on $W^{1}$ and has an expression
\begin{equation}
\Theta=\mathscr{L}\omega_{0}+\frac{\partial\mathscr{L}}{\partial y_{j}^{K}}\omega^{K}\wedge\omega_{j}.\label{eq:Poincare-Cartan}
\end{equation}
\end{cor}

\begin{thm}
\label{Thm:dLep}Let $\rho\in\Omega_{n}^{1}W$ be a Lepage form expressed
by \eqref{eq:LepFormExp}. Then
\[
(\pi^{2,1})^{*}d\rho=E+F,
\]
where $E$ is a $1$-contact, $\pi^{2,0}$-horizontal $(n+1)$-form,
which has a chart expression
\[
E=\left(\frac{\partial f_{0}}{\partial y^{K}}-d_{j}\frac{\partial f_{0}}{\partial y_{j}^{K}}\right)\omega^{K}\wedge\omega_{0},
\]
with $h\rho=f_{0}\omega_{0}$, and a form $F$ has the order of contactness
$\geq2$.
\end{thm}

\qquad{}In the class of Lepage equivalents of any \textit{first-order}
Lagrangian we have a possibility to determine a \textit{unique} Lepage
equivalent by means of additional requirements. Let $\lambda\in\Omega_{n,X}^{1}W$
be a Lagrangian, $\lambda=\mathscr{L}\omega_{0}$ with respect to
a fibered chart $(V,\psi)$, $\psi=(x^{i},y^{K})$, on $Y$. The following
theorems describe some well-known examples.
\begin{thm}
\textbf{\textup{\label{Theorem:Poincar-Cartan-equivalent}(Poincar\'e-Cartan
equivalent)}} There exists a unique Lepage equivalent $\Theta_{\lambda}\in\Omega_{n,Y}^{1}W$
of $\lambda$ such that the order of contactness of $\Theta{}_{\lambda}$
is $\leq1$. $\Theta{}_{\lambda}$ has a~local expression \eqref{eq:Poincare-Cartan}
with respect to a fibered chart $(V,\psi)$.
\end{thm}

\begin{thm}
\textbf{\textup{\label{Fundamental-Theorem}(Fundamental Lepage equivalent)}}
The differential $n$-form $Z_{\lambda}\in\Omega_{n}^{1}W$, given
in a chart $(V,\psi)$ by the expression
\begin{align}
Z_{\lambda} & =\mathscr{L}\omega_{0}+\sum_{k=1}^{n}\frac{1}{(n-k)!}\frac{1}{(k!)^{2}}\frac{\partial^{k}\mathscr{L}}{\partial y_{j_{1}}^{K_{1}}\partial y_{j_{2}}^{K_{2}}\ldots\partial y_{j_{k}}^{K_{k}}}\varepsilon_{j_{1}j_{2}\ldots j_{k}i_{k+1}i_{k+2}\ldots i_{n}}\label{eq:Fundamental}\\
 & \quad\cdot\omega^{K_{1}}\land\omega^{K_{2}}\wedge\ldots\wedge\omega^{K_{k}}\wedge dx^{i_{k+1}}\wedge dx^{i_{k+2}}\wedge\ldots\wedge dx^{i_{n}},\nonumber 
\end{align}
is a Lepage equivalent of the first-order Lagrangian $\lambda$.
\end{thm}

\begin{rem}
The crucial property of the fundamental Lepage equivalent is characterized
by two equivalent conditions: (i) $Z_{\lambda}$ is closed, (ii) $\lambda$
is trivial (i.e. the Euler\textendash Lagrange expressions associated
with $\lambda$ vanish identically). Moreover, if $\rho\in\Omega_{n}^{0}W$,
then $Z_{h\rho}=(\pi^{1,0})^{*}\rho$. This equivalent was discovered
by Krupka \cite{Krupka-Fund.Lep.eq.} (see also Betounes \cite{Betounes}).
Attempts to generalize the fundamental Lepage equivalent to higher-order
spaces appeared for dimension $n=2$ only, cf. Saunders and Crampin
\cite{Saunders-JGP}.

\end{rem}

\begin{thm}
\textbf{\textup{\label{Theorem:Caratheodory}(Carath\'eodory equivalent)
}}Let $\lambda\in\Omega_{n,X}^{1}W$ be a non-vanishing first-order
Lagrangian. Then the differential $n$-form $\Lambda_{\lambda}\in\Omega_{n}^{1}W$,
given in a chart $(V,\psi)$ by the expression
\begin{align}
\Lambda_{\lambda} & =\mathscr{L}\left(dx^{1}+\frac{1}{\mathscr{L}}\frac{\partial\mathscr{L}}{\partial y_{1}^{\sigma_{1}}}\omega^{\sigma_{1}}\right)\wedge\left(dx^{2}+\frac{1}{\mathscr{L}}\frac{\partial\mathscr{L}}{\partial y_{2}^{\sigma_{2}}}\omega^{\sigma_{2}}\right)\label{eq:CaratheodoryForm}\\
 & \qquad\wedge\ldots\wedge\left(dx^{n}+\frac{1}{\mathscr{L}}\frac{\partial\mathscr{L}}{\partial y_{n}^{\sigma_{n}}}\omega^{\sigma_{n}}\right),\nonumber 
\end{align}
is a Lepage equivalent of $\lambda$.
\end{thm}

\begin{rem}
The $n$-form $\Lambda_{\lambda}$ \eqref{eq:CaratheodoryForm} is
called the \textit{\textcolor{black}{Carath\'eodory form}}. It can
be shown that $\Lambda_{\lambda}$ is invariant with respect to \textit{all}
coordinate transformations on $Y$ (see Carath\'eodory \cite{Caratheodory},
Dedecker \cite{Dedecker}). 
\end{rem}

The next lemma is needed for further proofs.
\begin{lem}
\label{Lemma-koef}Suppose a $q$-form $\rho\in\Omega_{q}^{1}W$ has
a chart expression
\[
\rho=\sum_{k=0}^{q}\frac{1}{k!(q-k)!}B_{K_{1}\ldots K_{k}i_{k+1}\ldots i_{q}}\omega^{K_{1}}\wedge\ldots\wedge\omega^{K_{k}}\wedge dx^{i_{k+1}}\wedge\ldots\wedge dx^{i_{q}}
\]
with the coefficients skew-symmetric in all indices $K_{1},\:\ldots,\:K_{k}$,
and in all indices $i_{k+1},\:\ldots,\:i_{q}$. Then
\begin{equation}
\rho=\sum_{k=0}^{q}\frac{1}{k!(q-k)!}A_{K_{1}\ldots K_{k}i_{k+1}\ldots i_{q}}dy^{K_{1}}\wedge\ldots\wedge dy^{K_{k}}\wedge dx^{i_{k+1}}\wedge\ldots\wedge dx^{i_{q}},\label{eq:forma-dy}
\end{equation}
where
\begin{align}
A_{K_{1}\ldots K_{k}i_{k+1}\ldots i_{q}} & =\sum_{l=k}^{q}(-1)^{l-k}\left(_{q-l}^{q-k}\right)B_{K_{1}\ldots K_{l}i_{l+1}\ldots i_{q}}y_{i_{k+1}}^{K_{k+1}}\ldots y_{i_{l}}^{K_{l}}\label{eq:koeficienty-aux2}\\
 & \qquad\mathrm{Alt}(i_{k+1},\:\ldots,\:i_{q}).\nonumber 
\end{align}
\end{lem}

\begin{proof}
If a $q$-form $\rho\in\Omega_{q,Y}^{1}W$ has a chart expression
\eqref{eq:forma-dy}, then the contact components $p_{k}\rho$ of
$\rho$, where $0\leq k\leq q$, are given by
\[
p_{k}\rho=\frac{1}{k!(q-k)!}B_{K_{1}\ldots K_{k}i_{k+1}\ldots i_{q}}\omega^{K_{1}}\wedge\ldots\wedge\omega^{K_{k}}\wedge dx^{i_{k+1}}\wedge\ldots\wedge dx^{i_{q}},
\]
where
\begin{align}
B_{K_{1}\ldots K_{k}i_{k+1}\ldots i_{q}} & =\sum_{l=k}^{q}\left(_{q-l}^{q-k}\right)A_{K_{1}\ldots K_{l}i_{l+1}\ldots i_{q}}y_{i_{k+1}}^{K_{k+1}}\ldots y_{i_{l}}^{K_{l}}\;\mathrm{Alt}(i_{k+1},\:\ldots,\:i_{q})\label{eq:koeficienty-aux}
\end{align}
(a proof can be found in Krupka \cite{Krupka-Book} for differential
forms on arbitrary finite-order jet prolongation of a fibered manifold).
The identities \eqref{eq:koeficienty-aux} for all $k$, $0\leq k\leq q$,
constitute, however, a system of multi-linear equations which can
be directly solved with respect to the coefficients $A_{K_{1}\ldots K_{k}i_{k+1}\ldots i_{q}}$.
To obtain expressions \eqref{eq:koeficienty-aux2}, one can proceed
by induction with respect to degree $k$ of contactness of $\rho$.
\end{proof}

\section{First-order velocities and Grassmann fibrations \label{sec:3}}

From now on, the concepts and results of Section 2 will be employed
for the case of a~fibered manifold $Y=\mathbf{R}^{n}\times Q$, where
$Q$ is a smooth manifold of dimension $M=n+m$ for positive integers
$n$, $m$. The Cartesian coordinates of $\mathbf{R}^{n}$ are denoted
by $x^{i}$, $1\leq i\leq n$, and the canonical volume element of
$\mathbf{R}^{n}$ is denoted by $\omega_{0}=dx^{1}\wedge dx^{2}\wedge\ldots\wedge dx^{n}$.

Denote by $T_{n}^{1}Q$ the manifold of $n$-velocities of order $1$
over $Q$. Elements of $T_{n}^{1}Q$ are $1$-jets $J_{0}^{1}\zeta\in J_{(0,y)}^{1}(\mathbf{R}^{n},Q)$
with origin $0\in\mathbf{R}^{n}$ and target $y=\zeta(0)\in Q$. The
canonical projection $\tau_{n}^{1,0}:T_{n}^{1}Q\rightarrow Q$ is
defined by $\tau_{n}^{1,0}(J_{0}^{1}\zeta)=\zeta(0)$. In the standard
sense, $T_{n}^{1}Q$ is endowed with the canonical smooth manifold
structure: for any chart $(V,\psi)$, $\psi=(y^{K})$, $1\leq K\leq m+n$,
on $Q$, the pair $(V_{n}^{1},\psi_{n}^{1})$, $\psi_{n}^{1}=(y^{K},y_{j}^{K})$,
is a~chart on $T_{n}^{1}Q$, where $V_{n}^{1}=(\tau_{n}^{1,0})^{-1}(V)$,
$y_{j}^{K}(J_{0}^{1}\zeta)=D_{j}(y^{K}\zeta)(0)$, $1\leq j\leq n$,
and $\mathrm{dim}\,T_{n}^{1}Q=(n+m)(n+1)$. Recall that there is a
canonical identification of the jet space $J^{1}(\mathbf{R}^{n}\times Q)$
and the product $\mathbf{R}^{n}\times T_{n}^{1}Q$,
\begin{equation}
\phi:J^{1}(\mathbf{R}^{n}\times Q)\rightarrow\mathbf{R}^{n}\times T_{n}^{1}Q,\label{eq:Identification}
\end{equation}
defined by $\phi(J_{x}^{1}\gamma)=(x,J_{0}^{1}(\gamma_{0}\circ\mathrm{tr}_{-x}))$,
where $\gamma_{0}$ is the principal part of $\gamma$, and $\mathrm{tr}_{\alpha}:\mathbf{R}^{n}\rightarrow\mathbf{R}^{n}$,
$\mathrm{tr}_{\alpha}(x)=x-\alpha$, is the translation of the Euclidean
space $\mathbf{R}^{n}$. 

Suppose $\zeta:U\rightarrow Q$ be a differentiable mapping defined
on an open set $U\subset\mathbf{R}^{n}$. The $1$-\textit{jet prolongation}
of $\zeta$ is the mapping 
\begin{equation}
U\ni x\rightarrow(T_{n}^{1}\zeta)(x)=J_{0}^{1}(\zeta\circ\mathrm{tr}_{-x})\in T_{n}^{1}Q.\label{eq:JetProlong}
\end{equation}
Note that for any diffeomorphism $\mu:\bar{U}\rightarrow U$ of open
subsets of $\mathbf{R}^{n}$, the prolongation $T_{n}^{1}\zeta$ of
$\zeta$ satisfies
\begin{equation}
T_{n}^{1}(\zeta\circ\mu)(z)=T_{n}^{1}(\zeta)(\mu(z))\circ\mu^{1}(z),\label{eq:Prolongation}
\end{equation}
where $\mu^{1}(z)=J_{0}^{1}(\mathrm{tr}_{\mu(z)}\circ\mu\circ\mathrm{tr}_{-z})$.
Indeed, by the definition of $T_{n}^{1}\zeta$, $T_{n}^{1}(\zeta\circ\mu)(z)=J_{0}^{1}(\zeta\circ\mu\circ\mathrm{tr}_{-z})=J_{0}^{1}(\zeta\circ\mathrm{tr}_{-\mu(z)})\circ J_{0}^{1}(\mathrm{tr}_{\mu(z)}\circ\mu\circ\mathrm{tr}_{-z})=T_{n}^{1}(\zeta)(\mu(z))\circ\mu^{1}(z)$.
The identity \eqref{eq:Prolongation} is used to prove the forthcoming
Theorem \ref{Theorem:Zermelo}.

We restrict our attention to mappings which are \textit{immersions}
(i.e. their tangent mappings are injective). $J_{0}^{1}\zeta\in T_{n}^{1}Q$
is called \textit{regular}, if every representative of $J_{0}^{1}\zeta$
is an immersion at $0\in\mathbf{R}^{n}$. The set $\mathrm{Imm}\,T_{n}^{1}Q$
of regular velocities form an open subset of $T_{n}^{1}Q$, and with
the open submanifold structure $\mathrm{Imm}\,T_{n}^{1}Q$ is called
the \textit{manifold of regular} $n$-\textit{velocities} \textit{of
order} $1$ \textit{over} $Q$. Restricting the coordinates $\psi_{n}^{1}=(y^{K},y_{j}^{K})$,
we get the canonical charts on $\mathrm{Imm}\,T_{n}^{1}Q$, induced
by the canonical atlas of $T_{n}^{1}Q$. 

The manifold of regular velocities $\mathrm{Imm}\,T_{n}^{1}Q$ is
also endowed with another smooth structures. We denote by $(i)=(i_{1},i_{2},\ldots,i_{n})$
an increasing $n$-subsequence of the sequence $(1,2,\ldots,m+n)$,
and by $(\sigma)$ the complementary increasing subsequence. From
the definition of an immersion it follows that for every regular velocity
$J_{0}^{1}\zeta\in V_{n}^{1}$ there exists an $n$-subsequence $(i)$
of $(1,2,\ldots,m+n)$ such that $\mathrm{det}\,y_{j}^{i}(J_{0}^{1}\zeta)=\mathrm{det}\,(D_{j}(y^{i}\zeta)(0))\neq0$,
$i\in(i)$, $1\leq j\leq n$. We set 
\begin{equation}
V_{n}^{1(i)}=\{J_{0}^{1}\zeta\in V_{n}^{1}\,\,|\,\,\mathrm{det}\,(D_{j}(y^{i}\zeta)(0))\neq0\},\label{eq:(i)-neighborhood}
\end{equation}
for every $n$-subsequence $(i)$ of $(1,2,\ldots,m+n)$, where $(V_{n}^{1},\psi_{n}^{1})$
is a chart on $\mathrm{Imm}\,T_{n}^{1}Q$ associated with $(V,\psi)$.
Clearly, $V_{n}^{1(i)}$ is an open subset of $V_{n}^{1}$, and $V_{n}^{1}$
is covered by the sets $V_{n}^{1(i)}$, where $(i)$ runs through
all $n$-subsequences of $(1,2,\ldots,m+n)$. The charts $(V_{n}^{1(i)},\psi_{n}^{1(i)})$,
where $\psi_{n}^{1(i)}$ denotes the canonical coordinates $\psi_{n}^{1}$
on $V_{n}^{1}$ restricted to $V_{n}^{1(i)}$, constitute an atlas
on $\mathrm{Imm}\,T_{n}^{1}Q$ (\textit{\textcolor{black}{finer}}
than the canonical atlas). Another smooth atlas on $\mathrm{Imm}\,T_{n}^{1}Q$,
which arise from the canonical charts, is obtained in the following
way. To this purpose we introduce a set of functions $z_{i}^{k}:V_{n}^{1(i)}\rightarrow\mathbf{R}$,
defined by the formula $z_{i}^{k}y_{j}^{i}=\delta_{j}^{k}$ (the Kronecker
symbol), where $i\in(i)$, $1\leq j,k\leq n$. For every chart $(V_{n}^{1(i)},\psi_{n}^{1(i)})$,
$\psi_{n}^{1(i)}=(y^{i},y^{\sigma},y_{j}^{i},y_{j}^{\sigma})$, on
$\mathrm{Imm}\,T_{n}^{1}Q$ we put
\begin{equation}
w^{i}=y^{i},\quad w^{\sigma}=y^{\sigma},\quad w_{j}^{i}=y_{j}^{i},\quad w_{i}^{\sigma}=z_{i}^{j}y_{j}^{\sigma},\label{eq:AdaptCoord}
\end{equation}
where $i\in(i)$, $\sigma\in(\sigma)$, $1\leq j\leq n$. Formula
\eqref{eq:AdaptCoord} defines charts $(V_{n}^{1(i)},\chi_{n}^{1(i)})$,
$\chi_{n}^{1(i)}=(w^{i},w^{\sigma},w_{j}^{i},w_{i}^{\sigma})$, which
constitute an atlas on $\mathrm{Imm}\,T_{n}^{1}Q$. Note that the
coordinate functions $w_{i}^{\sigma}$ arise as the derivatives $\Delta_{i}w^{\sigma}$
of the base coordinates $w^{\sigma}$, where $\Delta_{i}$, $i\in(i)$,
is the $(i)$-\textit{adapted formal derivative morphism over $Q$,}
expressed by\textit{
\begin{align}
\Delta_{i} & =z_{i}^{j}d_{j}=z_{i}^{j}y_{j}^{K}\frac{\partial}{\partial y^{K}}=z_{i}^{j}y_{j}^{p}\frac{\partial}{\partial y^{p}}+z_{i}^{j}y_{j}^{\sigma}\frac{\partial}{\partial y^{\sigma}}=\frac{\partial}{\partial w^{i}}+w_{i}^{\sigma}\frac{\partial}{\partial w^{\sigma}},\label{eq:AdaptedDerivative}
\end{align}
}(summation runs through $1\leq j\leq n$, $K\in(1,2,\ldots,m+n)$).

A particular subset of regular $n$-velocities of order $1$ whose
both the origin and the target are at the point $0\in\mathbf{R}^{n}$,
$\mathrm{Imm}\,J_{(0,0)}^{1}(\mathbf{R}^{n},\mathbf{R}^{n})$, coincides
with the \textit{\textcolor{black}{general linear group}} $GL_{n}(\mathbf{R})$,
and is endowed with a global chart defined by the coordinate functions
$a_{j}^{i}:L_{n}^{1}\rightarrow\mathbf{R}$, $a_{j}^{i}(J_{0}^{1}\alpha)=D_{j}\alpha^{i}(0)$,
$1\leq i,j\leq n$. Note that for an arbitrary finite $r$, $L_{n}^{r}=\mathrm{Imm}\,J_{(0,0)}^{r}(\mathbf{R}^{n},\mathbf{R}^{n})$
is the $r$-\textit{th differential group} of $\mathbf{R}^{n}$ (known
from the theory of differential invariants), and $L_{n}^{1}=GL_{n}(\mathbf{R})$.
The canonical right action of $GL_{n}(\mathbf{R})$ on $T_{n}^{1}Q$
is defined by means of the jet composition and reduces onto $\mathrm{Imm}\,T_{n}^{1}Q$,
\begin{equation}
\mathrm{Imm}\,T_{n}^{1}Q\times GL_{n}(\mathbf{R})\ni(J_{0}^{1}\zeta,J_{0}^{1}\alpha)\rightarrow J_{0}^{1}\zeta\circ J_{0}^{1}\alpha=J_{0}^{1}(\zeta\circ\alpha)\in\mathrm{Imm}\,T_{n}^{1}Q.\label{eq:akce}
\end{equation}
It is easy to observe that the coordinates $w^{i},\,\,w^{\sigma},\,\,w_{i}^{\sigma}$
\eqref{eq:AdaptCoord} of the chart $(V_{n}^{1(i)},\chi_{n}^{1(i)})$
are $GL_{n}(\mathbf{R})$-invariant. $(V_{n}^{1(i)},\chi_{n}^{1(i)})$
is called the \textit{$(i)$-subordinate chart} to the chart $(V,\psi)$
on $Q$, \textit{\textcolor{black}{adapted}} to the canonical group
action of $GL_{n}(\mathbf{R})$ on $\mathrm{Imm}\,T_{n}^{1}Q$.

Consider the orbit space with respect to the group action \eqref{eq:akce},
\[
G_{n}^{1}Q=\mathrm{Imm}\,T_{n}^{1}Q/GL_{n}(\mathbf{R}),
\]
with its quotient topology. Elements of $G_{n}^{1}Q$ are classes
of regular velocities with respect to the equivalence relation ``\textit{there
exists an element $J_{0}^{1}\alpha\in GL_{n}(\mathbf{R})$ such that}
$J_{0}^{1}\zeta=J_{0}^{1}\chi\circ J_{0}^{1}\alpha$'' on $\mathrm{Imm}\,T_{n}^{1}Q$;
a class $[J_{0}^{1}\zeta]\in G_{n}^{1}Q$ is also called a \textit{contact
element} of order $1$ on $Q$, represented by an immersion $\zeta$.
The \textit{quotient projection} $\mathrm{Imm}\,T_{n}^{1}Q\ni J_{0}^{1}\zeta\rightarrow[J_{0}^{1}\zeta]\in G_{n}^{1}Q$
is denoted by $\kappa_{n}^{1}:\mathrm{Imm}\,T_{n}^{1}Q\rightarrow G_{n}^{1}Q$,
and the canonical projection $\tau_{n,G}^{1,0}:G_{n}^{1}Q\rightarrow Q$
is defined by $\tau_{n,G}^{1,0}([J_{0}^{1}\zeta])=\zeta(0)$.

The following theorem characterizes the structure of $\mathrm{Imm}\,T_{n}^{1}Q$
and $G_{n}^{1}Q$.
\begin{thm}
\label{Theorem:StructureGrassmann}Let $Q$ be Hausdorff.

\textup{\textcolor{black}{(a)}} The canonical action of $GL_{n}(\mathbf{R})$
defines on $\mathrm{Imm}\,T_{n}^{1}Q$ the structure of right principle
$GL_{n}(\mathbf{R})$-bundle with base $G_{n}^{1}Q$ and type fibre
$\mathrm{Imm}\,J_{(0,0)}^{r}(\mathbf{R}^{n},\mathbf{R}^{m+n})$.

\textup{\textcolor{black}{(b)}} The orbit space $G_{n}^{1}Q$ has
a unique smooth structure such that the canonical quotient projection
$\kappa_{n}^{1}$ is a submersion.

\textup{\textcolor{black}{(c)}} The orbit manifold $G_{n}^{1}Q$ has
the structure of a fibration with base $Q$, projection $\tau_{n,G}^{1,0},$
and type fibre $G_{n,n+m}^{1}=\mathrm{Imm}\,J_{(0,0)}^{r}(\mathbf{R}^{n},\mathbf{R}^{m+n})/GL_{n}(\mathbf{R})$.
The dimension of $G_{n}^{1}Q$ equals $\dim G_{n}^{1}Q=m(n+1)+n$.
\end{thm}

\begin{proof}
See Grigore and Krupka \cite{Grigore}.
\end{proof}
$G_{n}^{1}Q$ together with the smooth and fibration structure described
by Theorem \ref{Theorem:StructureGrassmann} is called the \textit{Grassmann
fibration of order} $1$ over the manifold $Q$. A smooth structure
on $G_{n}^{1}Q$ is described as follows. Let $(V_{n}^{1(i)},\chi_{n}^{1(i)})$,
$\chi_{n}^{1(i)}=(w^{i},w^{\sigma},w_{j}^{i},w_{i}^{\sigma})$, be
an $(i)$-subordinate chart on $\mathrm{Imm}\,T_{n}^{1}Q$ to the
chart $(V,\psi)$ on $Q$. Denote $\tilde{V}_{n}^{1(i)}=\kappa_{n}^{1}(V_{n}^{1(i)})$,
and $\tilde{\chi}_{n}^{1(i)}=(\tilde{w}^{i},\tilde{w}^{\sigma},\tilde{w}_{i}^{\sigma})$,
where 
\begin{equation}
\tilde{w}^{i}([J_{0}^{1}\zeta])=w^{i}(J_{0}^{1}\zeta),\quad\tilde{w}^{\sigma}([J_{0}^{1}\zeta])=w^{\sigma}(J_{0}^{1}\zeta),\quad\tilde{w}_{i}^{\sigma}([J_{0}^{1}\zeta])=w_{i}^{\sigma}(J_{0}^{1}\zeta).\label{eq:GrassmannChart}
\end{equation}
The chart $(\tilde{V}_{n}^{1(i)},\tilde{\chi}_{n}^{1(i)})$ on $G_{n}^{1}Q$
is the \textit{\textcolor{black}{associated}} chart with $(V_{n}^{1(i)},\chi_{n}^{1(i)})$.
Usually we do not distinguish between coordinates on $\mathrm{Imm}\,T_{n}^{1}Q$
and $G_{n}^{1}Q$, when no misunderstanding may arise.

Let $\zeta:U\rightarrow Q$ be an immersion on an open subset $U$
of $\mathbf{R}^{n}$ with values in $Q$. Recall that the 1-jet prolongation
of $\zeta$, the curve $T_{n}^{1}\zeta:U\rightarrow T_{n}^{1}Q$ is
defined by \eqref{eq:JetProlong}. An immersion $G_{n}^{1}\zeta:U\rightarrow G_{n}^{1}Q$,
given by
\begin{equation}
(G_{n}^{1}\zeta)(x)=[(T_{n}^{1}\zeta)(x)],\label{eq:GrassmannProlong}
\end{equation}
is called the \textit{Grassmann prolongation} of $\zeta$. If $\mu:\bar{U}\rightarrow U$
is a diffeomorphism of open subsets of $\mathbf{R}^{n}$, then the
mapping \eqref{eq:GrassmannProlong} satisfies
\[
G_{n}^{1}(\zeta\circ\mu)=(G_{n}^{1}\zeta)\circ\mu.
\]
Indeed, using the property \eqref{eq:Prolongation}, for any $z\in\bar{U}$
we have
\begin{align*}
G_{n}^{1}(\zeta\circ\mu)(z) & =[T_{n}^{1}(\zeta\circ\mu)(z)]=[T_{n}^{1}(\zeta)(\mu(z))\circ\mu^{1}(z)]\\
 & =[T_{n}^{1}(\zeta)(\mu(z))]=G_{n}^{1}(\zeta)(\mu(z)).
\end{align*}

An arbitrary diffeomorphism $\alpha:W\rightarrow Q$ (onto its image),
defined on an open subset $W\subset Q$, can be prolonged on the Grassmann
fibration $G_{n}^{1}Q$ as follows. We define a diffeomorphism of
$G_{n}^{1}Q$, $G_{n}^{1}\alpha:W^{1}\rightarrow G_{n}^{1}Q$, where
$\tilde{W}^{1}=(\tau_{n,G}^{1,0})^{-1}(W)$, by putting
\begin{equation}
G_{n}^{1}\alpha([J_{0}^{1}\zeta])=[J_{0}^{1}(\alpha\circ\zeta)]\label{eq:GrassmannProlongationDiffeo}
\end{equation}
for every $J_{0}^{1}\zeta\in\mathrm{Imm}\,T_{n}^{1}Q$ such that $\alpha$
is composable with $\zeta$. $G_{n}^{1}\alpha$ is called the \textit{Grassmann
prolongation} of $\alpha$. Note that combining these concepts, the
prolongation of an immersion and the prolongation of a diffeomorphism,
it is easy to see that $G_{n}^{1}(\alpha\circ\zeta)=G_{n}^{1}\alpha\circ G_{n}^{1}\zeta$;
for every $x\in U$,
\begin{equation}
G_{n}^{1}(\alpha\circ\zeta)(x)=[J_{0}^{1}(\alpha\circ\zeta\circ\mathrm{tr}_{-x})]=G_{n}^{1}\alpha([J_{0}^{1}(\zeta\circ\mathrm{tr}_{-x})])=G_{n}^{1}\alpha\circ G_{n}^{1}\zeta(x).\label{eq:ProlongationProperty}
\end{equation}

Consider a vector field $\Xi$ defined on $W\subset Q$, and the local
one-parameter group $\alpha_{t}^{\Xi}$ of $\Xi$. $G_{n}^{1}\alpha_{t}^{\Xi}$
is the Grassmann prolongation of $\alpha_{t}^{\Xi}$ \eqref{eq:GrassmannProlongationDiffeo}.
We define
\begin{equation}
G_{n}^{1}\Xi([J_{0}^{1}\zeta])=\left(\frac{d}{dt}G_{n}^{1}\alpha_{t}^{\Xi}([J_{0}^{1}\zeta])\right)_{0},\label{eq:GrassmannProlongationField}
\end{equation}
for every element $[J_{0}^{1}\zeta]\in\tilde{W}^{1}\subset G_{n}^{1}Q$.
Formula \eqref{eq:GrassmannProlongationField} defines a vector field
$G^{1}\Xi$ on an open subset $W^{1}$, called the \textit{Grassmann
prolongation} of vector field $\Xi$. 
\begin{lem}
\label{Lem:VectorFieldProlong}If a vector field $\Xi$ on $W\subset Q$
has an expression
\[
\Xi=\Xi^{K}\frac{\partial}{\partial y^{K}}
\]
 with respect to a chart $(V,\psi)$, $\psi=(y^{K})$, such that $V\subset W$,
then the Grassmann prolongation $G_{n}^{1}\Xi$ is expressed with
respect to the $(i)$-subordinate chart $(\tilde{V}_{n}^{1(i)},\tilde{\chi}_{n}^{1(i)})$
\eqref{eq:GrassmannChart} as 
\[
G_{n}^{1}\Xi=\Xi^{i}\frac{\partial}{\partial w^{i}}+\Xi^{\sigma}\frac{\partial}{\partial w^{\sigma}}+\Xi_{i}^{\sigma}\frac{\partial}{\partial w_{i}^{\sigma}},
\]
where $\Xi^{K}=\Xi^{K}(w^{i},w^{\sigma})$, $\Xi_{i}^{\sigma}=\Xi_{i}^{\sigma}(w^{i},w^{\sigma},w_{i}^{\sigma})$,
and 
\begin{align}
\Xi_{i}^{\sigma} & =\Delta_{i}\Xi^{\sigma}-w_{p}^{\sigma}\Delta_{i}\Xi^{p}.\label{eq:ProlongFieldExp}
\end{align}
\end{lem}

\begin{proof}
By the definition of $G_{n}^{1}\Xi$ \eqref{eq:GrassmannProlongationField},
\[
\Xi_{i}^{\sigma}([J_{0}^{1}\zeta])=\left(\frac{d}{dt}w_{i}^{\sigma}\circ G_{n}^{1}\alpha_{t}^{\Xi}([J_{0}^{1}\zeta])\right)_{0},
\]
where we differentiate the function $t\rightarrow w_{i}^{\sigma}\circ G_{n}^{1}\alpha_{t}^{\Xi}([J_{0}^{1}\zeta])=\Delta_{i}w^{\sigma}(J_{0}^{1}(\alpha_{t}^{\Xi}\circ\zeta))$
at $t=0$, where $\Delta_{i}=z_{i}^{j}d_{j}$ is the $(i)$-adapted
formal derivative morphism \eqref{eq:AdaptedDerivative}\textit{.}
Using the chart transformation \eqref{eq:AdaptCoord}, a straightforward
calculation at a point $J_{0}^{1}\zeta$ now gives
\begin{align*}
 & \Xi_{i}^{\sigma}([J_{0}^{1}\zeta])=\left(\frac{d}{dt}z_{i}^{j}(J_{0}^{1}(\alpha_{t}^{\Xi}\circ\zeta))\right)_{0}d_{j}w^{\sigma}+z_{i}^{j}\left(\frac{d}{dt}d_{j}w^{\sigma}(J_{0}^{1}(\alpha_{t}^{\Xi}\circ\zeta))\right)_{0}\\
 & =\frac{\partial z_{i}^{j}}{\partial y_{k}^{p}}\left(\frac{d}{dt}D_{k}(y^{p}\alpha_{t}^{\Xi}\psi^{-1}\circ\psi\zeta))\right)_{0}y_{j}^{\sigma}+z_{i}^{j}\left(\frac{d}{dt}D_{j}(y^{\sigma}\alpha_{t}^{\Xi}\psi^{-1}\circ\psi\zeta)(0))\right)_{0}\\
 & =-z_{p}^{j}z_{i}^{k}y_{j}^{\sigma}y_{k}^{L}\left(\frac{d}{dt}D_{L}(y^{p}\alpha_{t}^{\Xi}\psi^{-1})(\psi\zeta(0))\right)_{0}+z_{i}^{j}y_{j}^{L}\left(\frac{d}{dt}D_{L}(y^{\sigma}\alpha_{t}^{\Xi}\psi^{-1})(\psi\zeta(0))\right)_{0}\\
 & =-\delta_{i}^{s}w_{p}^{\sigma}\left(\frac{\partial\Xi^{p}}{\partial y^{s}}\right)_{\zeta(0)}-w_{i}^{\nu}w_{p}^{\sigma}\left(\frac{\partial\Xi^{p}}{\partial y^{\nu}}\right)_{\zeta(0)}+\delta_{i}^{s}\left(\frac{\partial\Xi^{\sigma}}{\partial y^{s}}\right)_{\zeta(0)}+w_{i}^{\nu}\left(\frac{\partial\Xi^{\sigma}}{\partial y^{\nu}}\right)_{\zeta(0)}\\
 & =\left(\frac{\partial\Xi^{\sigma}}{\partial y^{i}}\right)_{\zeta(0)}+w_{i}^{\nu}\left(\frac{\partial\Xi^{\sigma}}{\partial y^{\nu}}\right)_{\zeta(0)}-w_{p}^{\sigma}\left(\frac{\partial\Xi^{p}}{\partial y^{i}}\right)_{\zeta(0)}-w_{p}^{\sigma}w_{i}^{\nu}\left(\frac{\partial\Xi^{p}}{\partial y^{\nu}}\right)_{\zeta(0)}\\
 & =\Delta_{i}\Xi^{\sigma}-w_{p}^{\sigma}\Delta_{i}\Xi^{p},
\end{align*}
proving \eqref{eq:ProlongFieldExp}.
\end{proof}
\begin{rem}
Lemma \ref{Lem:VectorFieldProlong} is a modification of the standard
concept of a jet prolongation of a projectable vector field, defined
on a fibered manifold. If $\Xi$ is a vector field on $T_{n}^{1}Q$
(embedded in $\mathbf{R}^{n}\times T_{n}^{1}Q$), expressed by $\Xi=\Xi^{K}(\partial/\partial y^{K})$,
then the jet prolongation $J^{1}\Xi$, resp. $J^{2}\Xi$, of $\Xi$
has the expression
\[
J^{1}\Xi=\Xi^{K}\frac{\partial}{\partial y^{K}}+\Xi_{j}^{K}\frac{\partial}{\partial y_{j}^{K}},\;\textrm{\;resp.}\;\;J^{2}\Xi=\Xi^{K}\frac{\partial}{\partial y^{K}}+\Xi_{j}^{K}\frac{\partial}{\partial y_{j}^{K}}+\Xi_{jl}^{K}\frac{\partial}{\partial y_{jl}^{K}},
\]
where $\Xi_{j}^{K}=d_{j}\Xi^{K}$, $\Xi_{jl}^{K}=d_{l}\Xi_{j}^{K}$.
\end{rem}

A differential form $\eta$ on $\tilde{W}^{1}\subset G_{n}^{1}Q$
is called \textit{\textcolor{black}{contact}}, if $(G_{n}^{1}\zeta)^{*}\eta=0$
for all immersions $\zeta:U\rightarrow W$, defined on an open subset
$U$ of $\mathbf{R}^{n}$. If $(V,\psi)$, $\psi=(y^{K})$, is a chart
on $Q$ such that $V\subset W$, and $(\tilde{V}_{n}^{1(i)},\tilde{\chi}_{n}^{1(i)})$,
$\tilde{\chi}_{n}^{1(i)}=(w^{i},w^{\sigma},w_{i}^{\sigma})$, is the
$(i)$-subordinate chart on $G_{n}^{1}Q$, then the $1$-forms $dw^{i}$,
$\tilde{\omega}^{\sigma}$, $dw_{i}^{\sigma}$, where 
\begin{equation}
\tilde{\omega}^{\sigma}=dw^{\sigma}-w_{i}^{\sigma}dw^{i}\label{eq:OmegaGrassmann}
\end{equation}
(sum through $i\in(i)$), constitute a basis of linear forms on $\tilde{V}_{n}^{1(i)}$.
Differential forms, which are locally generated by contact $1$-forms
$\tilde{\omega}^{\sigma}$ and $2$-forms $d\tilde{\omega}^{\sigma}$,
constitute the \textit{\textcolor{black}{contact ideal}} $\tilde{\varTheta}^{1}W$
in the exterior algebra of differential forms on $\tilde{W}^{1}$.
Using the definition of charts on $G_{n}^{1}Q$, and the canonical
embedding of $\mathrm{Imm}\,T_{n}^{1}Q$ into $J^{1}(\mathbf{R}^{n}\times Q)$,
we note that $\tilde{\omega}^{\sigma}$ can be expressed as a linear
combination of contact $1$-forms $\omega^{K}=dy^{K}-y_{l}^{K}dx^{l}$
\eqref{eq:omega} on $J^{1}(\mathbf{R}^{n}\times Q)$,
\begin{align*}
(\pi^{2,1})^{*}\tilde{\omega}^{\sigma} & =(\pi^{2,1})^{*}(dy^{\sigma}-z_{i}^{j}y_{j}^{\sigma}dy^{i})\\
 & =\omega^{\sigma}+y_{l}^{\sigma}dx^{l}-z_{i}^{j}y_{j}^{\sigma}(\omega^{i}+y_{l}^{i}dx^{l})=\omega^{\sigma}-z_{i}^{j}y_{j}^{\sigma}\omega^{i}.
\end{align*}

\section{The fundamental Lepage equivalent of a homogeneous Lagrangian \label{sec:3-1}}

First we recall the notion of a positive homogeneous function and
study its properties. This is defined as an equivariance with respect
to the canonical right action of the identity component $GL_{n}^{+}(\mathbf{R})$
of the general linear group $GL_{n}(\mathbf{R})$. We say that a real-valued
function $F:\mathrm{Imm}\,T_{n}^{1}Q\rightarrow\mathbf{R}$ is \textit{positive
homogeneous}, if 
\[
F(J_{0}^{1}\zeta\circ J_{0}^{1}\alpha)=\mathrm{det}(a_{j}^{i}(J_{0}^{1}\alpha))F(J_{0}^{1}\zeta)
\]
for all $J_{0}^{1}\zeta\in\mathrm{Imm}\,T_{n}^{1}Q$ and all $J_{0}^{1}\alpha\in GL_{n}^{+}(\mathbf{R})$.
Note that the elements of $GL_{n}(\mathbf{R})$ used in this definition
are represented by \textit{orientation-preserving} diffeomorphisms.

Let $U$ be an open subset of $\mathbf{R}^{n}$, and let $\zeta:U\rightarrow Q$
be an immersion. Any compact subset $S$ of $U$ associates with a
function $F:\mathrm{Imm}\,T_{n}^{1}Q\rightarrow\mathbf{R}$ the integral

\begin{equation}
F_{S}(\zeta)=\int_{S}(F\circ T_{n}^{1}\zeta)\omega_{0}.\label{eq:Integral}
\end{equation}

The following theorem is a criterion of parameter-invariance of this
integral.
\begin{thm}
\textbf{\textup{(Euler\textendash Zermelo)}} \label{Theorem:Zermelo}Let
$F:\mathrm{Imm}\,T_{n}^{1}Q\rightarrow\mathbf{R}$ be a differentiable
function. The following conditions are equivalent: 

\textup{(a)} $F$ is positive homogeneous. 

\textup{(b)} The integral \eqref{eq:Integral} is parameter-invariant,
i.e. if $\zeta:U\rightarrow Q$ is an immersion, and $\mu:\bar{U}\rightarrow U$
is a diffeomorphism of open subsets in $\mathbf{R}^{n}$ such that
$\det D\mu>0$ on $\bar{U}$, then for any two compact subsets $\bar{U}_{0}\subset\bar{U}$,
$U_{0}\subset U$, such that $\mu(\bar{U}_{0})=U_{0}$, the integral
\eqref{eq:Integral} satisfies $F_{U_{0}}(\zeta)=F_{\bar{U}_{0}}(\zeta\circ\mu)$.

\textup{(c)} For any chart $(V,\psi)$, $\psi=(y^{K})$, on $Q$,

\begin{equation}
\frac{\partial F}{\partial y_{j}^{K}}y_{l}^{K}=\delta_{l}^{j}F,\quad j,l=1,2,\ldots,n.\label{eq:Zermelo}
\end{equation}
\end{thm}

\begin{proof}
The proof for $n=1$ can be found in Urban and Krupka \cite{UK-Debrecen};
for arbitrary positive integer $n$ the proof proceeds along the same
lines.
\end{proof}
\begin{rem}
The identities \eqref{eq:Zermelo} are the well-known \textit{Zermelo
conditions} (cf. McKiernan \cite{McKiernan}). Differentiating \eqref{eq:Zermelo}
with respect to $y_{j}^{K}$, we obtain the following formulas for
partial derivatives of a positive homogeneous function $F$. Namely,
for an integer $k\geq1$ we have
\begin{align}
 & \frac{\partial^{k}F}{\partial y_{j_{1}}^{K_{1}}\partial y_{j_{2}}^{K_{2}}\ldots\partial y_{j_{k}}^{K_{k}}}y_{i_{1}}^{K_{1}}=\frac{\partial^{k-1}F}{\partial y_{j_{2}}^{K_{2}}\partial y_{j_{3}}^{K_{3}}\ldots\partial y_{j_{k}}^{K_{k}}}\delta_{i_{1}}^{j_{1}}-\frac{\partial^{k-1}F}{\partial y_{j_{1}}^{K_{2}}\partial y_{j_{3}}^{K_{3}}\ldots\partial y_{j_{k}}^{K_{k}}}\delta_{i_{1}}^{j_{2}}\label{eq:ZermeloDifferentiating}\\
 & \quad\quad\quad\quad\quad\quad-\frac{\partial^{k-1}F}{\partial y_{j_{2}}^{K_{2}}\partial y_{j_{1}}^{K_{3}}\partial y_{j_{4}}^{K_{4}}\ldots\partial y_{j_{k}}^{K_{k}}}\delta_{i_{1}}^{j_{3}}-\ldots-\frac{\partial^{k-1}F}{\partial y_{j_{2}}^{K_{2}}\partial y_{j_{3}}^{K_{3}}\ldots\partial y_{j_{1}}^{K_{k}}}\delta_{i_{1}}^{j_{k}}.\nonumber 
\end{align}
\end{rem}

\begin{rem}
\label{rem:G-Zermelo}On the chart neighborhood $V_{n}^{1(i)}$ \eqref{eq:(i)-neighborhood},
the Zermelo conditions \eqref{eq:Zermelo} for the function $\tilde{F}=F\circ(\psi_{n}^{1(i)})^{-1}\circ\chi_{n}^{1(i)}$
reads
\[
\frac{\partial\tilde{F}}{\partial w_{j}^{i}}w_{k}^{i}=\delta_{k}^{j}\tilde{F},
\]
where summation runs through $i\in(i)$. It is easy to see that this
equation can be integrated and its solution is of the form 
\begin{equation}
\tilde{F}(w^{i},w^{\sigma},w_{j}^{i},w_{i}^{\sigma})=\mathrm{det}(w_{j}^{i})\tilde{F}_{G}(w^{i},w^{\sigma},w_{i}^{\sigma}),\label{eq:G-function}
\end{equation}
where $\tilde{F}_{G}$ is a uniquely given differentiable function
on $\tilde{V}_{n}^{1(i)}\subset G_{n}^{1}Q$. We call $\tilde{F}_{G}$
the \textit{Grassmann projection} of the (positive homogeneous) function
$F$, associated to $(\tilde{V}_{n}^{1(i)},\tilde{\chi}_{n}^{1(i)})$.
\end{rem}

\begin{lem}
\label{Lemma:vlastnosti-homog}Let $F:\mathrm{Imm}\,T_{n}^{1}Q\rightarrow\mathbf{R}$
be a positive homogeneous function. Then for every integer $k$, $0\leq k\leq n-1$,
and every integer $l$ such that, $k<l\leq n$, 

\begin{align}
\frac{\partial^{l}F}{\partial y_{j_{1}}^{K_{1}}\ldots\partial y_{j_{l}}^{K_{l}}}y_{i_{k+1}}^{K_{k+1}}\ldots y_{i_{l}}^{K_{l}}\varepsilon_{j_{1}\ldots j_{l}i_{l+1}\ldots i_{n}} & =\frac{l!}{k!}\frac{\partial^{k}F}{\partial y_{j_{1}}^{K_{1}}\ldots\partial y_{j_{k}}^{K_{k}}}\varepsilon_{j_{1}\ldots j_{k}i_{k+1}\ldots i_{n}}.\label{eq:DerivativesHomogeneous}
\end{align}
\end{lem}

\begin{proof}
Formula \eqref{eq:DerivativesHomogeneous} follows from a direct application
of \eqref{eq:ZermeloDifferentiating} in $(l-k)$ steps, and the skew-symmetry
property of the Levi-Civita symbol. Indeed, we have

\begin{align*}
 & \frac{\partial^{l}F}{\partial y_{j_{1}}^{K_{1}}\ldots y_{j_{k}}^{K_{k}}y_{j_{k+1}}^{K_{k+1}}\ldots\partial y_{j_{l}}^{K_{l}}}y_{i_{k+1}}^{K_{k+1}}\ldots y_{i_{l}}^{K_{l}}\varepsilon_{j_{1}\ldots j_{l}i_{l+1}\ldots i_{n}}\\
 & =l\frac{\partial^{l-1}F}{\partial y_{j_{1}}^{K_{1}}\ldots y_{j_{k}}^{K_{k}}y_{j_{k+1}}^{K_{k+1}}\ldots\partial y_{j_{l-1}}^{K_{l-1}}}y_{i_{k+1}}^{K_{k+1}}\ldots y_{i_{l-1}}^{K_{l-1}}\varepsilon_{j_{1}\ldots j_{l-1}i_{l}\ldots i_{n}}\\
 & =l(l-1)\frac{\partial^{l-2}F}{\partial y_{j_{1}}^{K_{1}}\ldots y_{j_{k}}^{K_{k}}y_{j_{k+1}}^{K_{k+1}}\ldots\partial y_{j_{l-2}}^{K_{l-2}}}y_{i_{k+1}}^{K_{k+1}}\ldots y_{i_{l-2}}^{K_{l-2}}\varepsilon_{j_{1}\ldots j_{l-2}i_{l-1}i_{l}\ldots i_{n}}\\
 & =\ldots=l(l-1)(l-2)\ldots(k+2)(k+1)\frac{\partial^{k}F}{\partial y_{j_{1}}^{K_{1}}\ldots\partial y_{j_{k}}^{K_{k}}}\varepsilon_{j_{1}\ldots j_{k}i_{k+1}\ldots i_{n}},
\end{align*}
as required.
\end{proof}
In the following theorem we study the fundamental Lepage equivalent
$Z_{\lambda}$ \eqref{eq:Fundamental} of a~first-order Lagrangian
$\lambda\in\Omega_{n,\mathbf{R}^{n}}^{1}(\mathbf{R}^{n}\times Q)$,
$\lambda=\mathscr{L}\omega_{0}$, whose Lagrange function $\mathscr{L}$,
defined on the manifold of regular velocities $\mathrm{Imm}\,T_{n}^{1}Q$,
is positive-homogeneous. We derive its local structure and show, in
particular, that the resulting $n$-form is defined on the Grassmann
fibration $G_{n}^{1}Q$.
\begin{thm}
\label{Theorem:LepageHilbert}Let $\lambda\in\Omega_{n,\mathbf{R}^{n}}^{1}(\mathbf{R}^{n}\times Q)$
be a Lagrangian of order $1$, expressed by $\lambda=\mathscr{L}\omega_{0}$,
where $\mathscr{L}:\mathrm{Imm}\,T_{n}^{1}Q\rightarrow\mathbf{R}$
is the Lagrange function. If $\mathscr{L}$ satisfies the Zermelo
condition \eqref{eq:Zermelo}, then the fundamental Lepage equivalent
\eqref{eq:Fundamental} is defined on the $n$-Grassmann fibration
$G_{n}^{1}Q$, and has the expression
\begin{equation}
W_{\lambda}=\frac{1}{(n!)^{2}}\frac{\partial^{n}\mathscr{L}}{\partial y_{j_{1}}^{K_{1}}\partial y_{j_{2}}^{K_{2}}\ldots\partial y_{j_{n}}^{K_{n}}}\varepsilon_{j_{1}j_{2}\ldots j_{n}}dy^{K_{1}}\land dy^{K_{2}}\wedge\ldots\wedge dy^{K_{n}}.\label{eq:FundLepEqZermelo}
\end{equation}
\end{thm}

\begin{proof}
Suppose $\mathscr{L}$ satisfies the Zermelo condition \eqref{eq:Zermelo}
hence also identities \eqref{eq:DerivativesHomogeneous}. By Lemma
\ref{Lemma-koef}, the fundamental Lepage equivalent \eqref{eq:Fundamental}
can be expressed as

\begin{align}
Z_{\lambda} & =\sum_{k=0}^{n}\frac{1}{(n-k)!}\frac{1}{k!}Z_{K_{1}\ldots K_{k}i_{k+1}\ldots i_{n}}dy^{K_{1}}\wedge\ldots\wedge dy^{K_{k}}\wedge dx^{i_{k+1}}\wedge\ldots\wedge dx^{i_{n}},\label{eq:FundLepEq-2}
\end{align}
where
\begin{align}
Z_{K_{1}\ldots K_{k}i_{k+1}\ldots i_{n}} & =\sum_{l=k}^{n}(-1)^{l-k}\left(_{n-l}^{n-k}\right)\frac{1}{l!}\frac{\partial^{l}\mathscr{L}}{\partial y_{j_{1}}^{K_{1}}\ldots\partial y_{j_{l}}^{K_{l}}}\varepsilon_{j_{1}\ldots j_{l}i_{l+1}\ldots i_{n}}\label{eq:FundLepEqCoef}\\
 & \qquad\cdot y_{i_{k+1}}^{K_{k+1}}\ldots y_{i_{l}}^{K_{l}}\,\mathrm{Alt}(i_{k+1},\:\ldots,\:i_{n}).\nonumber 
\end{align}
We claim that $Z_{\lambda}$ coincides with $W_{\lambda}$ \eqref{eq:FundLepEqZermelo}.
Clearly, the term of $Z_{\lambda}$ containing the exterior product
$dy^{K_{1}}\wedge\ldots\wedge dy^{K_{n}}$ reads
\begin{align*}
 & \frac{1}{n!}Z_{K_{1}\ldots K_{n}}dy^{K_{1}}\wedge\ldots\wedge dy^{K_{n}}=\frac{1}{n!}\left(\frac{1}{n!}\frac{\partial^{n}\mathscr{L}}{\partial y_{j_{1}}^{K_{1}}\ldots\partial y_{j_{n}}^{K_{n}}}\varepsilon_{j_{1}\ldots j_{n}}\right)dy^{K_{1}}\wedge\ldots\wedge dy^{K_{n}},
\end{align*}
which coincides with \eqref{eq:FundLepEqZermelo}. Let us show that
all terms of $Z_{\lambda}$ \eqref{eq:FundLepEq-2}, which contain
exactly $k$ exterior factors $dy^{K}$ vanish whenever $0\leq k\leq n-1$.
Indeed, consider the term of $Z_{\lambda}$ containing the exterior
product $dy^{K_{1}}\wedge\ldots\wedge dy^{K_{k}}$ for fixed $k$,
$0\leq k\leq n-1$, whose coefficients are given by \eqref{eq:FundLepEqCoef}.
Applying Lemma \ref{Lemma:vlastnosti-homog}, we obtain

\begin{align*}
Z_{K_{1}\ldots K_{k}i_{k+1}\ldots i_{n}} & =\sum_{l=k}^{n}(-1)^{l-k}\left(_{n-l}^{n-k}\right)\frac{1}{l!}\frac{l!}{k!}\frac{\partial^{k}\mathscr{L}}{\partial y_{j_{1}}^{K_{1}}\partial y_{j_{2}}^{K_{2}}\ldots\partial y_{j_{k}}^{K_{k}}}\varepsilon_{j_{1}\ldots j_{k}i_{k+1}\ldots i_{n}}\\
 & =\frac{1}{k!}\frac{\partial^{k}\mathscr{L}}{\partial y_{j_{1}}^{K_{1}}\partial y_{j_{2}}^{K_{2}}\ldots\partial y_{j_{k}}^{K_{k}}}\varepsilon_{j_{1}\ldots j_{k}i_{k+1}\ldots i_{n}}\sum_{l=k}^{n}(-1)^{l-k}\left(_{n-l}^{n-k}\right)=0,
\end{align*}
as required.\textcolor{black}{{} }

\textcolor{black}{It remains to show that }$W_{\lambda}$ \eqref{eq:FundLepEqZermelo}\textcolor{black}{{}
is defined on }$G_{n}^{1}Q$.\textcolor{black}{{} This is, however,
a consequence of positive homogeneity of the Lagrange function }$\mathscr{L}$\textcolor{black}{{}
and properties of the coordinate transformation }$\psi_{n}^{1(i)}\circ(\chi_{n}^{1(i)})^{-1}$\textcolor{black}{.
Indeed, expressing the $n$-form \eqref{eq:FundLepEqZermelo} in terms
of the} $GL_{n}(\mathbf{R})$\textcolor{black}{-adapted }$(i)$\textcolor{black}{-subordinate
chart }$(V_{n}^{1(i)},\chi_{n}^{1(i)})$ on $\mathrm{Imm}\,T_{n}^{1}Q$,
arbitrary term of $W_{\lambda}$, up to a coefficient, can be written
as
\begin{align*}
 & \frac{\partial^{n-l}\mathscr{\tilde{L}}}{\partial w_{i_{l+1}}^{\sigma_{l+1}}\ldots\partial w_{i_{k}}^{\sigma_{k}}\partial w_{i_{k+1}}^{\sigma_{k+1}}\ldots\partial w_{i_{n}}^{\sigma_{n}}}\left|\begin{array}{cccc}
z_{i_{1}}^{j_{1}} & z_{i_{2}}^{j_{1}} &  & z_{i_{n}}^{j_{1}}\\
z_{i_{1}}^{j_{2}} & z_{i_{2}}^{j_{2}} &  & z_{i_{n}}^{j_{2}}\\
 &  & \ddots\\
z_{i_{1}}^{j_{n}} & z_{i_{2}}^{j_{n}} &  & z_{i_{n}}^{j_{n}}
\end{array}\right|\\
 & w_{p_{l+1}}^{\sigma_{l+1}}\ldots w_{p_{k}}^{\sigma_{k}}\delta_{p_{1}}^{i_{1}}\ldots\delta_{p_{l}}^{i_{l}}\varepsilon_{j_{1}j_{2}\ldots j_{n}}dw^{p_{1}}\land\ldots\wedge dw^{p_{k}}\land\tilde{\omega}^{\sigma_{k+1}}\land\ldots\wedge\tilde{\omega}^{\sigma_{n}},
\end{align*}
for some integers $k$, $l$ such that $0\leq k\leq n$, $0\leq l\leq k$,
where $\mathscr{\tilde{L}}=\mathscr{L}\circ(\psi_{n}^{1(i)})^{-1}\circ\chi_{n}^{1(i)}$,
and $\tilde{\omega}^{\sigma}=dw^{\sigma}-w_{i}^{\sigma}dw^{i}$ \eqref{eq:OmegaGrassmann}
are contact 1-forms on $\tilde{V}_{n}^{1(i)}$ (cf. Urban and Krupka
\cite{UK-Acta}). Since the determinant $\mathrm{det}(z_{i}^{j})$,
where $i\in(i),$ $1\leq j\leq n$, does not depend on the coordinates
$w_{i}^{\sigma}$, we obtain by Remark \ref{rem:G-Zermelo}, \eqref{eq:G-function},
that the function $\mathscr{\tilde{L}}\mathrm{det}(z_{i}^{j})$ coincides
with the Grassmann projection $\tilde{\mathscr{L}}_{G}$ of $\mathscr{L}$.
Hence $W_{\lambda}$ is defined on the neighborhood $\tilde{V}_{n}^{1(i)}\subset G_{n}^{1}Q$
for every chart $(V,\psi)$ on $Q$. This completes the proof.
\end{proof}
\begin{rem}
\label{Remark:Hilbert-Caratheodory}(\textit{\textcolor{black}{Hilbert-Carath\'eodory
form}}) Analogously, Crampin and Saunders \cite{CramSaun1} studied
the Carath\'eodory form \eqref{eq:CaratheodoryForm} for a positive-homogeneous
non-vanishing Lagrange function, resulting into the \textcolor{black}{Hilbert-Carath\'eodory
form, expressed by}
\begin{equation}
\Lambda_{\lambda}=\frac{1}{n!}\frac{1}{\mathscr{L}^{n-1}}\frac{\partial\mathscr{L}}{\partial y_{j_{1}}^{K_{1}}}\frac{\partial\mathscr{L}}{\partial y_{j_{2}}^{K_{2}}}\ldots\frac{\partial\mathscr{L}}{\partial y_{j_{n}}^{K_{n}}}\varepsilon_{j_{1}j_{2}\ldots j_{n}}dy^{K_{1}}\land dy^{K_{2}}\wedge\ldots\wedge dy^{K_{n}}.\label{eq:Hilbert-CaratheodoryForm}
\end{equation}
We point out that $\Lambda_{\lambda}$ \eqref{eq:Hilbert-CaratheodoryForm}
is \textit{\textcolor{black}{different}} from the $n$-form $W_{\lambda}$
\eqref{eq:FundLepEqZermelo}: a simple example for $n=2$ is the Lagrangian
$\lambda=\mathscr{L}\omega_{0}$ with positive-homogeneous function
$\mathscr{L}=A_{PQ}y_{1}^{P}y_{2}^{Q}$, where functions $A_{PQ}=A_{PQ}(y^{K})$
are skew-symmetric; see also Crampin and Saunders \cite{CramSaun3}.
On the other hand, these two Lepage equivalents are closely related
as they coincide for the Lagrangian of minimal submanifolds problem
(see Section \ref{sec:6}).
\end{rem}

\section{First-order variational field theory for submanifolds \label{sec:5}}

\subsection{Lepage forms on $\mathrm{Imm}\,T_{n}^{1}Q$ and the Euler\textendash Lagrange
form\label{sec:5-1}}

Based on the canonical identification of $J^{1}(\mathbf{R}^{n}\times Q)$
and $\mathbf{R}^{n}\times T_{n}^{1}Q$ (see \eqref{eq:Identification}),
we say that an $n$-form $\rho$ on $\mathrm{Imm}\,T_{n}^{1}Q$ is
a \textit{Lepage form}, if $\rho$ is a Lepage form on $J^{1}(\mathbf{R}^{n}\times Q)$,
i.e. if $hi_{\xi}d\rho=0$ for an arbitrary $\pi^{1,0}$-vertical
vector field $\xi$ on $(\pi^{1,0})^{-1}(W)$, where $W\subset\mathbf{R}^{n}\times Q$
is an open set, and $\pi^{1,0}:J^{1}(\mathbf{R}^{n}\times Q)\rightarrow\mathbf{R}^{n}\times Q$
is the canonical jet projection. 

Restrict ourselves to $\pi^{1,0}$-\textit{horizontal} $n$-forms
on $\mathrm{Imm}\,T_{n}^{1}Q$, expressed in any chart $(V,\psi)$
on $Q$ by
\begin{equation}
\rho=\frac{1}{n!}A_{K_{1}K_{2}\ldots K_{n}}dy^{K_{1}}\land dy^{K_{2}}\wedge\ldots\wedge dy^{K_{n}},\label{eq:horizontalForm}
\end{equation}
where the coefficients $A_{K_{1}K_{2}\ldots K_{n}}$ are functions
on $\mathrm{Imm}\,T_{n}^{1}Q$. Using the canonical decomposition
formula $\rho=h\rho+\sum_{k=1}^{n}p_{k}\rho$, we get
\begin{equation}
h\rho=\mathscr{L}\omega_{0},\label{eq:LocalLagrangian}
\end{equation}
where
\begin{equation}
\mathscr{L}=\frac{1}{n!}A_{K_{1}\ldots K_{n}}y_{j_{1}}^{K_{1}}\ldots y_{j_{n}}^{K_{n}}\varepsilon^{j_{1}\ldots j_{n}}\label{eq:LagrangianLepage}
\end{equation}
is a function on $\mathrm{Imm}\,T_{n}^{1}Q$, and
\begin{equation}
p_{k}\rho=\frac{1}{k!(n-k)!}B_{K_{1}\ldots K_{k}j_{k+1}\ldots j_{n}}\omega^{K_{1}}\wedge\ldots\wedge\omega^{K_{k}}\wedge dx^{j_{k+1}}\wedge\ldots\wedge dx^{j_{n}},\label{eq:k-contact-component}
\end{equation}
where
\[
B_{K_{1}\ldots K_{k}j_{k+1}\ldots j_{n}}=A_{K_{1}\ldots K_{n}}y_{j_{k+1}}^{K_{k+1}}\ldots y_{j_{n}}^{K_{n}}\quad\mathrm{Alt}(j_{k+1},\:\ldots,\:j_{n}).
\]
The $n$-form $h\rho$ \eqref{eq:LocalLagrangian} is said to be the
\textit{Lagrangian}, and its component $\mathscr{L}$ \eqref{eq:LagrangianLepage}
the\textit{ Lagrange function}, associated to $\rho$. If $\rho$
is a Lepage form, $\rho$ is also said to be a\textit{~Lepage equivalent}
of the Lagrangian $\lambda=h\rho$.

Now we give a characterization of Lepage forms on $\mathrm{Imm}\,T_{n}^{1}Q$.
\begin{thm}
\label{Theorem:LepCond}Let $(V,\psi)$ be a chart on $Q$. Suppose
$\rho$ is a $\pi^{1,0}$-horizontal $n$-form on $\mathrm{Imm}\,T_{n}^{1}Q$,
expressed by \eqref{eq:horizontalForm}. Then $\rho$ is a Lepage
form if and only if
\begin{align}
\frac{\partial A_{K_{1}K_{2}\ldots K_{n}}}{\partial y_{s}^{P}}y_{j_{1}}^{K_{1}}y_{j_{2}}^{K_{2}}\ldots y_{j_{n}}^{K_{n}} & \varepsilon^{j_{1}j_{2}\ldots j_{n}}=0.\label{eq:LepageCondition}
\end{align}
\end{thm}

\begin{proof}
We apply the criterion for Lepage forms on $J^{1}(\mathbf{R}^{n}\times Q)$.
If $\rho$ is expressed by \eqref{eq:horizontalForm}, we get from
\eqref{eq:k-contact-component}, in particular,

\begin{align*}
p_{1}\rho & =\frac{1}{1!(n-1)!}A_{K_{1}\ldots K_{n}}y_{j_{2}}^{K_{2}}\ldots y_{j_{n}}^{K_{n}}\omega^{K_{1}}\wedge dx^{j_{2}}\wedge\ldots\wedge dx^{j_{n}}\\
 & =\frac{1}{1!(n-1)!}A_{KK_{2}\ldots K_{n}}y_{j_{2}}^{K_{2}}\ldots y_{j_{n}}^{K_{n}}\varepsilon^{jj_{2}\ldots j_{n}}\omega^{K}\wedge\omega_{j},
\end{align*}
where $\omega^{K}=dy^{K}-y_{j}^{K}dx^{j}$ and $\omega_{j}=i_{\partial/\partial x^{j}}\omega_{0}$.
By Theorem \ref{Theorem:LepageForms} the form $\rho$ is a Lepage
form, or the Lepage equivalent of $h\rho=\mathscr{L}\omega_{0}$,
where $\mathscr{L}$ is given by \eqref{eq:LagrangianLepage}, if
and only if the principal component $\Theta$ of $\rho$ has the expression
\[
\Theta=\mathscr{L}\omega_{0}+\frac{\partial\mathscr{L}}{\partial y_{j}^{K}}\omega^{K}\wedge\omega_{j}.
\]
This means, however, that $\rho$ is a Lepage form if and only if
\begin{equation}
\frac{1}{1!(n-1)!}A_{KK_{2}\ldots K_{n}}y_{j_{2}}^{K_{2}}\ldots y_{j_{n}}^{K_{n}}\varepsilon^{jj_{2}\ldots j_{n}}=\frac{\partial\mathscr{L}}{\partial y_{j}^{K}}.\label{eq:CompareDerivative}
\end{equation}
Differentiating $\mathscr{L}$ \eqref{eq:LagrangianLepage}, we have
\begin{align*}
 & \frac{\partial\mathscr{L}}{\partial y_{j}^{K}}=\frac{1}{n!}\frac{\partial A_{K_{1}\ldots K_{n}}}{\partial y_{j}^{K}}y_{j_{1}}^{K_{1}}\ldots y_{j_{n}}^{K_{n}}\varepsilon^{j_{1}\ldots j_{n}}+\frac{1}{n!}A_{K_{1}\ldots K_{n}}\frac{\partial}{\partial y_{j}^{K}}\left(y_{j_{1}}^{K_{1}}\ldots y_{j_{n}}^{K_{n}}\right)\varepsilon^{j_{1}\ldots j_{n}}\\
 & \;=\frac{1}{n!}\frac{\partial A_{K_{1}\ldots K_{n}}}{\partial y_{j}^{K}}y_{j_{1}}^{K_{1}}\ldots y_{j_{n}}^{K_{n}}\varepsilon^{j_{1}\ldots j_{n}}\\
 & \;+\frac{1}{n!}A_{KK_{2}\ldots K_{n}}y_{j_{2}}^{K_{2}}\ldots y_{j_{n}}^{K_{n}}\varepsilon^{jj_{2}\ldots j_{n}}+\frac{1}{n!}A_{K_{1}KK_{3}\ldots K_{n}}y_{j_{1}}^{K_{1}}y_{j_{3}}^{K_{3}}\ldots y_{j_{n}}^{K_{n}}\varepsilon^{j_{1}jj_{3}\ldots j_{n}}\\
 & \;+\ldots+\frac{1}{n!}A_{K_{1}\ldots K_{n-1}K}y_{j_{1}}^{K_{1}}\ldots y_{j_{n-1}}^{K_{n-1}}\varepsilon^{j_{1}\ldots j_{n-1}j}\\
 & \;=\frac{1}{n!}\frac{\partial A_{K_{1}\ldots K_{n}}}{\partial y_{j}^{K}}y_{j_{1}}^{K_{1}}\ldots y_{j_{n}}^{K_{n}}\varepsilon^{j_{1}\ldots j_{n}}+\frac{1}{(n-1)!}A_{KK_{2}\ldots K_{n}}y_{j_{2}}^{K_{2}}\ldots y_{j_{n}}^{K_{n}}\varepsilon^{jj_{2}\ldots j_{n}}.
\end{align*}
Substituting now this expression on the right-hand side of \eqref{eq:CompareDerivative},
we get the formula \eqref{eq:LepageCondition}. This completes the
proof.
\end{proof}
\begin{cor}
\label{Corollary:LagrangeFunctionHomog}If $\rho$ is a Lepage form
on $\mathrm{Imm}\,T_{n}^{1}Q$, then $h\rho=\mathscr{L}\omega_{0}$,
where the function $\mathscr{L}:\mathrm{Imm}\,T_{n}^{1}Q\rightarrow\mathbf{R}$
is positive homogeneous.
\end{cor}

\begin{proof}
Suppose $\rho$ is expressed by \eqref{eq:horizontalForm} with respect
to a chart $(V,\psi)$ on $Q$, hence $\mathscr{L}$ has an expression
\eqref{eq:LagrangianLepage}. It is easy to verify the Zermelo conditions
\eqref{eq:Zermelo}. Indeed, applying the Lepage form criterion \eqref{eq:LepageCondition},
we have
\begin{align}
\frac{\partial\mathscr{L}}{\partial y_{j}^{K}}y_{k}^{K} & =\frac{1}{(n-1)!}A_{K_{1}K_{2}\ldots K_{n}}y_{k}^{K_{1}}y_{j_{2}}^{K_{2}}\ldots y_{j_{n}}^{K_{n}}\varepsilon^{jj_{2}\ldots j_{n}}.\label{eq:ZermeloSpecial}
\end{align}
Since $A_{K_{1}K_{2}\ldots K_{n}}$(resp. $\varepsilon^{jj_{2}\ldots j_{n}}$)
are skew-symmetric in subscripts (resp. superscripts), we see that
the right-hand side of \eqref{eq:ZermeloSpecial} vanishes whenever
$j\neq k$. From \eqref{eq:ZermeloSpecial} and \eqref{eq:LagrangianLepage}
it follows that
\[
\frac{\partial\mathscr{L}}{\partial y_{j}^{K}}y_{j}^{K}=\frac{1}{(n-1)!}A_{K_{1}K_{2}\ldots K_{n}}y_{j}^{K_{1}}y_{j_{2}}^{K_{2}}\ldots y_{j_{n}}^{K_{n}}\varepsilon^{jj_{2}\ldots j_{n}}=n\mathscr{L},
\]
as required.
\end{proof}
Applying Theorem \ref{Theorem:LepCond} and Corollary \ref{Corollary:LagrangeFunctionHomog},
we obtain a mapping $\rho\rightarrow W_{h\rho}$, assigning to any
Lepage form $\rho$ on $\mathrm{Imm}\,T_{n}^{1}Q$ the Lepage equivalent
$W_{h\rho}$ of $h\rho$ defined on $G_{n}^{1}Q$ by formula \eqref{eq:FundLepEqZermelo}.
\begin{thm}
\label{CorollaryLepage}\textup{(a)} For any Lagrangian $\lambda\in\Omega_{n,\mathbf{R}^{n}}^{1}(\mathbf{R}^{n}\times Q)$,
$W_{\lambda}$ is a Lepage form.

\textup{(b)} If $\rho$ is a Lepage form on $\mathrm{Imm}\,T_{n}^{1}Q$,
then the fundamental Lepage equivalent $W_{h\rho}$ of $h\rho$ satisfies
\[
(T_{n}^{1}\zeta)^{*}W_{h\rho}=(T_{n}^{1}\zeta)^{*}\rho,
\]
for arbitrary immersion $\zeta:U\rightarrow Q$. 

\textup{(c)} Let $\rho$ be a Lepage form on $\mathrm{Imm}\,T_{n}^{1}Q$,
expressed by \eqref{eq:horizontalForm}. Then $\rho$ coincides with
the fundamental Lepage equivalent $W_{h\rho}$ of $h\rho$ if and
only if the components of $\rho$ satisfy
\[
\frac{\partial^{k}A_{K_{1}\ldots K_{n-k}L_{n-k+1}\ldots L_{n}}}{\partial y_{j_{n-k+1}}^{K_{n-k+1}}\ldots\partial y_{j_{n}}^{K_{n}}}y_{j_{n-k+1}}^{L_{n-k+1}}\ldots y_{j_{n}}^{L_{n}}=0,
\]
for every $k$, $1\leq k\leq n-1$.
\end{thm}

\begin{proof}
1. Assertion (a) is a direct consequence of the definition of $W_{\lambda}$.
Indeed, one can also verify the Lepage form condition \eqref{eq:LepageCondition}
directly for
\[
A_{K_{1}K_{2}\ldots K_{n}}=\frac{1}{n!}\frac{\partial^{n}\mathscr{L}}{\partial y_{j_{1}}^{K_{1}}\partial y_{j_{2}}^{K_{2}}\ldots\partial y_{j_{n}}^{K_{n}}}\varepsilon_{j_{1}j_{2}\ldots j_{n}}.
\]
2. Since both $\rho$ and $W_{h\rho}$ are Lepage equivalents of the
Lagrangian $h\rho$, the assertion (b) is implied by the definition
of Lepage equivalent of a Lagrangian. 

However, we prove (b) directly since a proof of assertion (c) is included.
Suppose $h\rho=\mathscr{L}\omega_{0}$, where $\mathscr{L}:\mathrm{Imm}\,T_{n}^{1}Q\rightarrow\mathbf{R}$
is the Lagrange function expressed by \eqref{eq:LagrangianLepage}.
Applying the fact that $\rho$ is a Lepage form satisfying \eqref{eq:LepageCondition},
we obtain for any $1\leq p\leq n$, 
\begin{align}
\frac{\partial^{p}\mathscr{L}}{\partial y_{j_{1}}^{K_{1}}\ldots\partial y_{j_{p}}^{K_{p}}} & =\sum_{k=0}^{p-1}\frac{1}{(n-p+k)!}\left(\begin{array}{c}
p-1\\
k
\end{array}\right)\frac{\partial^{k}A_{K_{1}\ldots K_{p-k}L_{p-k+1}\ldots L_{n}}}{\partial y_{j_{p-k+1}}^{K_{p-k+1}}\ldots\partial y_{j_{p}}^{K_{p}}}\label{eq:nthDerivativeL-1}\\
 & \quad\quad\quad\cdot y_{s_{p-k+1}}^{L_{p-k+1}}\ldots y_{s_{n}}^{L_{n}}\varepsilon^{j_{1}\ldots j_{p-k}s_{p-k+1}\ldots s_{n}}.\nonumber 
\end{align}
Hence $W_{h\rho}$, given by \eqref{eq:FundLepEqZermelo}, is expressed
as
\begin{align}
W_{h\rho} & =\frac{1}{(n!)^{2}}\frac{\partial^{n}\mathscr{L}}{\partial y_{j_{1}}^{K_{1}}\partial y_{j_{2}}^{K_{2}}\ldots\partial y_{j_{n}}^{K_{n}}}\varepsilon_{j_{1}j_{2}\ldots j_{n}}dy^{K_{1}}\land dy^{K_{2}}\wedge\ldots\wedge dy^{K_{n}}\nonumber \\
 & =\frac{1}{(n!)^{2}}\left(\sum_{k=0}^{n-1}\frac{1}{k!}\left(\begin{array}{c}
n-1\\
k
\end{array}\right)\frac{\partial^{k}A_{K_{1}\ldots K_{n-k}L_{n-k+1}\ldots L_{n}}}{\partial y_{j_{n-k+1}}^{K_{n-k+1}}\ldots\partial y_{j_{n}}^{K_{n}}}y_{s_{n-k+1}}^{L_{n-k+1}}\ldots y_{s_{n}}^{L_{n}}\right)\nonumber \\
 & \quad\quad\quad\cdot\varepsilon^{j_{1}\ldots j_{n-k}s_{n-k+1}\ldots s_{n}}\varepsilon_{j_{1}j_{2}\ldots j_{n}}dy^{K_{1}}\land dy^{K_{2}}\wedge\ldots\wedge dy^{K_{n}}\label{eq:Whrho}\\
 & =\frac{1}{(n!)^{2}}\left(\sum_{k=1}^{n-1}\frac{1}{k!}\left(\begin{array}{c}
n-1\\
k
\end{array}\right)\frac{\partial^{k}A_{K_{1}\ldots K_{n-k}L_{n-k+1}\ldots L_{n}}}{\partial y_{j_{n-k+1}}^{K_{n-k+1}}\ldots\partial y_{j_{n}}^{K_{n}}}y_{j_{n-k+1}}^{L_{n-k+1}}\ldots y_{j_{n}}^{L_{n}}(n-k)!k!\right.\nonumber \\
 & \quad\quad\quad+n!A_{K_{1}\ldots K_{n}}\Biggr)\cdot dy^{K_{1}}\land dy^{K_{2}}\wedge\ldots\wedge dy^{K_{n}}.\nonumber 
\end{align}
Applying the identities \eqref{eq:ZermeloDifferentiating} into \eqref{eq:nthDerivativeL-1},
and using \eqref{eq:nthDerivativeL-1} recursively for $1\leq p\leq n$,
we derive the condition
\begin{equation}
\frac{\partial^{k}A_{K_{1}\ldots K_{n-k}L_{n-k+1}\ldots L_{n}}}{\partial y_{j_{n-k+1}}^{K_{n-k+1}}\ldots\partial y_{j_{n-1}}^{K_{n-1}}\partial y_{j_{n}}^{K_{n}}}y_{i_{1}}^{K_{1}}\ldots y_{i_{n-k}}^{K_{n-k}}y_{i_{n-k+1}}^{K_{n-k+1}}\ldots y_{i_{n-1}}^{K_{n-1}}y_{j_{n-k+1}}^{L_{n-k+1}}\ldots y_{j_{n}}^{L_{n}}=0\label{eq:ConditionDerived}
\end{equation}
for all $k,$ where $1\leq k\leq n-1$. Now, with the help of \eqref{eq:ConditionDerived},
it is straightforward to see that the pull-back operation along $T_{n}^{1}\zeta$
onto $W_{h\rho}$ \eqref{eq:Whrho} gives 
\[
T_{n}^{1}\zeta{}^{*}W_{h\rho}=T_{n}^{1}\zeta{}^{*}\left(\frac{1}{n!}A_{K_{1}\ldots K_{n}}dy^{K_{1}}\land dy^{K_{2}}\wedge\ldots\wedge dy^{K_{n}}\right)=T_{n}^{1}\zeta{}^{*}\rho.
\]
3. Assertion (c) directly follows from the proof of (b), the expression
\eqref{eq:Whrho} of $W_{h\rho}$.
\end{proof}
The following theorem describes a close relation of Lepage forms with
the calculus of variations, well-known from the variational theory
on fibered manifolds (cf. Krupka and Saunders \cite{Krupka-Book},
and references therein).
\begin{thm}
Suppose $\rho$ is a $\pi^{1,0}$-horizontal Lepage form on $\mathrm{Imm}\,T_{n}^{1}Q$,
and $\mathscr{L}:\mathrm{Imm}\,T_{n}^{1}Q\rightarrow\mathbf{R}$ is
the Lagrange function associated to $\rho$. Then in any chart $(V,\psi)$
on $Q$
\begin{equation}
p_{1}d\rho=E_{K}(\mathscr{L})\omega^{K}\land\omega_{0},\label{eq:Euler-Lagrange-form}
\end{equation}
where
\begin{equation}
E_{K}(\mathscr{L})=\frac{\partial\mathscr{L}}{\partial y^{K}}-d_{j}\frac{\partial\mathscr{L}}{\partial y_{j}^{K}}.\label{eq:E-L-expressions}
\end{equation}
\end{thm}

\begin{proof}
Suppose $\rho$ has an expression \eqref{eq:horizontalForm} with
respect to a chart on $Q$, with coefficients satisfying the Lepage
condition \eqref{eq:LepageCondition}. We compute the exterior derivative
$d\rho$ of the Lepage form $\rho$, and its $1$-contact component
$p_{1}d\rho$. We have
\[
d\rho=\frac{1}{(n+1)!}\left(B_{L_{1}L_{2}\ldots L_{n+1}}dy^{L_{1}}+B_{L_{1}L_{2}\ldots L_{n+1}}^{j}dy_{j}^{L_{1}}\right)\land dy^{L_{2}}\wedge\ldots\wedge dy^{L_{n+1}},
\]
where
\begin{align*}
B_{L_{1}L_{2}\ldots L_{n+1}} & =\frac{\partial A_{L_{2}L_{3}\ldots L_{n+1}}}{\partial y^{L_{1}}}-\frac{\partial A_{L_{1}L_{3}L_{4}\ldots L_{n+1}}}{\partial y^{L_{2}}}-\ldots-\frac{\partial A_{L_{2}L_{3}\ldots L_{n}L_{1}}}{\partial y^{L_{n+1}}},\\
B_{L_{1}L_{2}\ldots L_{n+1}}^{j} & =(n+1)\frac{\partial A_{L_{2}L_{3}\ldots L_{n+1}}}{\partial y_{j}^{L_{1}}}.
\end{align*}
By a direct calculation we get
\begin{align}
p_{1}d\rho & =\frac{1}{n!}B_{KL_{1}L_{2}\ldots L_{n}}y_{i_{1}}^{L_{1}}y_{i_{2}}^{L_{2}}\ldots y_{i_{n}}^{L_{n}}\varepsilon^{i_{1}i_{2}\ldots i_{n}}\omega^{K}\wedge\omega_{0}\nonumber \\
 & +\frac{1}{(n+1)!}B_{KL_{1}L_{2}\ldots L_{n}}^{j}y_{i_{1}}^{L_{1}}y_{i_{2}}^{L_{2}}\ldots y_{i_{n}}^{L_{n}}\varepsilon^{i_{1}i_{2}\ldots i_{n}}\omega_{j}^{K}\land\omega_{0}\label{eq:p1dRho}\\
 & -\frac{n}{(n+1)!}B_{L_{1}KL_{2}\ldots L_{n}}^{j}y_{ji_{1}}^{L_{1}}y_{i_{2}}^{L_{2}}\ldots y_{i_{n}}^{L_{n}}\varepsilon^{i_{1}i_{2}\ldots i_{n}}\omega^{K}\land\omega_{0}.\nonumber 
\end{align}
But from the condition \eqref{eq:LepageCondition} it follows that
the second summand of \eqref{eq:p1dRho} vanishes hence $p_{1}d\rho=\varepsilon_{K}\omega^{K}\wedge\omega_{0}$,
where
\begin{align}
\varepsilon_{K} & =\frac{1}{n!}\left(\frac{\partial A_{L_{1}L_{2}\ldots L_{n}}}{\partial y^{K}}-\frac{\partial A_{KL_{2}L_{3}\ldots L_{n}}}{\partial y^{L_{1}}}-\ldots-\frac{\partial A_{L_{1}L_{2}\ldots L_{n-1}K}}{\partial y^{L_{n}}}\right)\label{eq:epsPom}\\
 & \quad\cdot y_{i_{1}}^{L_{1}}y_{i_{2}}^{L_{2}}\ldots y_{i_{n}}^{L_{n}}\varepsilon^{i_{1}i_{2}\ldots i_{n}}-\frac{1}{(n-1)!}\frac{\partial A_{KL_{2}\ldots L_{n}}}{\partial y_{j}^{L_{1}}}y_{ji_{1}}^{L_{1}}y_{i_{2}}^{L_{2}}\ldots y_{i_{n}}^{L_{n}}\varepsilon^{i_{1}i_{2}\ldots i_{n}}.\nonumber 
\end{align}
Now, using the properties of the Lagrange function $\mathscr{L}$
associated to $\rho$, it is easy to verify that the coefficients
of $p_{1}d\rho$, the functions $\varepsilon_{K}$ \eqref{eq:epsPom},
coincide with $E_{K}(\mathscr{L})$ given by \eqref{eq:E-L-expressions}.
\end{proof}
The $(n+1)$-form $E_{\rho}=p_{1}d\rho$ \eqref{eq:Euler-Lagrange-form}
is called the \textit{Euler\textendash Lagrange form} associated with
a Lepage form $\rho$, and its coefficients $E_{K}(\mathscr{L})$
are the \textit{Euler\textendash Lagrange expressions} of the Lagrange
function $\mathscr{L}$. 

\subsection{The variational integral, variations, and the first variation formula}

Let $\rho$ be a $\pi^{1,0}$-horizontal  Lepage form on $\mathrm{Imm}\,T_{n}^{1}Q$,
and $h\rho=\mathscr{L}\omega_{0}$ the Lagrangian associated to $\rho$,
where $\mathscr{L}:\mathrm{Imm}\,T_{n}^{1}Q\rightarrow\mathbf{R}$
is the Lagrange function given by \eqref{eq:LagrangianLepage}. Let
$\zeta:U\rightarrow Q$ be an immersion on an open subset $U$ of
$\mathbf{R}^{n}$. Any compact subset $\Omega$ of $U$ associates
with $\rho$ the integral
\begin{equation}
\rho_{\Omega}(\zeta)=\int_{\Omega}(T_{n}^{1}\zeta)^{*}\rho=\int_{\Omega}(\mathscr{L}\circ T_{n}^{1}\zeta)\omega_{0}.\label{eq:Integral-rho}
\end{equation}
Since $\mathscr{L}$ is positive homogeneous (Corollary \ref{Corollary:LagrangeFunctionHomog}),
it follows from Theorem \ref{Theorem:Zermelo} that the integral \eqref{eq:Integral-rho}
is invariant with respect to reparametrizations; this means that
\[
\rho_{\Omega_{1}}(\zeta)=\rho_{\Omega_{2}}(\zeta\circ\mu)
\]
for any orientation-preserving diffeomorphism $\mu$ of $\mathbf{R}^{n}$
and any two compact subsets $\Omega_{1},\Omega_{2}\subset U\subset\mathbf{R}^{n}$
such that $\mu(\Omega_{2})=\Omega_{1}$. Moreover, by Theorem \ref{CorollaryLepage},
(b), it follows that $\rho$ and $W_{h\rho}$ \eqref{eq:FundLepEqZermelo}
define the same integral \eqref{eq:Integral-rho}. Thus, we have
\begin{equation}
\rho_{\Omega}(\zeta)=\int_{\Omega}(G_{n}^{1}\zeta)^{*}W_{h\rho},\label{eq:Integral-W}
\end{equation}
where $G_{n}^{1}\zeta$ is the Grassmann prolongation of $\zeta$
(see \eqref{eq:GrassmannProlong}).

The real-valued function $\zeta\rightarrow\rho_{\Omega}(\zeta)$,
defined on the set of immersions from an open subset $U$ of $\mathbf{R}^{n}$
into $Q$, is called the \textit{variational functional} associated
with $\rho$ and $\Omega$. In the standard sense, we describe the
behaviour of $\rho_{\Omega}$ with respect to variations (deformations)
of $\zeta$ or, more precisely, variations of the image $\zeta(\Omega)\subset Q$.
Following the geometric foundations of the calculus of variations
on fibered spaces (see Krupka \cite{Krupka-GlobalFunctionals}), variations
of $\zeta$ are induced by means of a vector field.

Let $\Xi$ be a vector field on an open subset $W\subset Q$, and
$\alpha_{t}^{\Xi}$ be its local one-parameter group. For every immersion
$\zeta:\mathbf{R}^{n}\supset U\rightarrow Q$ such that $\zeta(U)\subset W$,
$\alpha_{t}^{\Xi}\circ\zeta$ is a one-parameter family of immersions,
which depends smoothly on parameter $t$, and it is called a \textit{variation}
of $\zeta$ induced by the vector field $\Xi$. Since $\Omega$ is
supposed to be compact, the value of the variational integral $\rho_{\Omega}(\alpha_{t}^{\Xi}\circ\zeta)$
\eqref{eq:Integral-rho}, \eqref{eq:Integral-W}, is defined for all
sufficiently small $t$, and we get
\[
\rho_{\Omega}(\alpha_{t}^{\Xi}\circ\zeta)=\int_{\Omega}(G_{n}^{1}(\alpha_{t}^{\Xi}\circ\zeta))^{*}W_{h\rho}=\int_{\Omega}G_{n}^{1}\zeta{}^{*}G_{n}^{1}\alpha_{t}^{\Xi}{}^{*}W_{h\rho},
\]
where the identity \eqref{eq:ProlongationProperty} is applied. Appling
the theorem on differentiating an integral dependent on a parameter,
and with the help of the concept of a Lie derivative of a differential
form, we obtain
\[
\left(\frac{d}{dt}\rho_{\Omega}(\alpha_{t}^{\Xi}\circ\zeta)\right)_{0}=\int_{\Omega}G_{n}^{1}\zeta{}^{*}\left(\frac{d}{dt}G_{n}^{1}\alpha_{t}^{\Xi}{}^{*}W_{h\rho}\right)_{0}=\int_{\Omega}G_{n}^{1}\zeta{}^{*}\partial_{G_{n}^{1}\Xi}W_{h\rho}.
\]
With the notation given by \eqref{eq:Integral-W}, we may denote this
number as
\begin{equation}
(\partial_{G_{n}^{1}\Xi}W_{h\rho})_{\Omega}(\zeta)=\int_{\Omega}G_{n}^{1}\zeta{}^{*}\partial_{G_{n}^{1}\Xi}W_{h\rho}.\label{eq:Variation}
\end{equation}
\eqref{eq:Variation} is called the \textit{variation} of the variational
functional $\zeta\rightarrow\rho_{\Omega}(\zeta)$ \eqref{eq:Integral-W}
at a point $\zeta$, induced by the vector field $\Xi$. The corresponding
functional $\zeta\rightarrow(\partial_{G_{n}^{1}\Xi}W_{h\rho})_{\Omega}(\zeta)$
is the \textit{variational derivative} of $\rho_{\Omega}$ with respect
to $\Xi$.
\begin{thm}
\label{Thm:VarFormula}Let $\rho$ be a $\pi^{1,0}$-horizontal Lepage
form on $\mathrm{Imm}\,T_{n}^{1}Q$, and $W_{h\rho}$ be the fundamental
Lepage equivalent of the Lagrangian $h\rho$. Let $\Xi$ be a vector
field on an open subset $W\subset Q$. Then the variational derivative
$\partial_{G_{n}^{1}\Xi}W_{h\rho}$ of $\rho_{\Omega}$ satisfies:

\textup{(a)} For every immersion $\zeta:\mathbf{R}^{n}\supset U\rightarrow Q$
such that $\zeta(U)\subset W$,
\begin{equation}
G_{n}^{1}\zeta{}^{*}\partial_{G_{n}^{1}\Xi}W_{h\rho}=J^{2}(\mathrm{id}_{\mathbf{R}^{n}}\times\zeta)^{*}i_{J^{2}\Xi}E_{\rho}+d(G_{n}^{1}\zeta^{*}i_{G_{n}^{1}\Xi}W_{h\rho}).\label{eq:FirstVarFormula}
\end{equation}

\textup{(b)} For every compact submanifold $\Omega$ of $U\subset\mathbf{R}^{n}$,
and every immersion $\zeta:\mathbf{R}^{n}\supset U\rightarrow Q$
such that $\zeta(U)\subset W$,
\begin{equation}
\int_{\Omega}G_{n}^{1}\zeta{}^{*}\partial_{G_{n}^{1}\Xi}W_{h\rho}=\int_{\Omega}J^{2}(\mathrm{id}{}_{\mathbf{R}^{n}}\times\zeta)^{*}i_{J^{2}\Xi}E_{\rho}+\int_{\partial\Omega}G_{n}^{1}\zeta^{*}i_{G_{n}^{1}\Xi}W_{h\rho}.\label{eq:IntegralVarFormula}
\end{equation}
\end{thm}

\begin{proof}
(a) Suppose $\Xi$ is a vector field on $W\subset Q$. Using the Cartan's
formula, we write the Lie derivative of $W_{h\rho}$ with respect
to $G_{n}^{1}\Xi$ as
\[
\partial_{G_{n}^{1}\Xi}W_{h\rho}=i_{G_{n}^{1}\Xi}dW_{h\rho}+di_{G_{n}^{1}\Xi}W_{h\rho}.
\]
Since $W_{h\rho}$ and $\rho$ are both Lepage equivalents of the
Lagrangian $h\rho$, it follows from Theorem \ref{Thm:dLep} that
$(\pi^{2,1}){}^{*}dW_{h\rho}=E_{\rho}+F$, where $E_{\rho}$ is the
Euler\textendash Lagrange form \eqref{eq:Euler-Lagrange-form}, and
$F$ has the order of contactness $\geq2$. For any form $\eta$ defined
on $G_{n}^{1}Q$, we have $(\pi^{2,1}){}^{*}i_{G_{n}^{1}\Xi}\eta=i_{J^{2}\Xi}(\pi^{2,1}){}^{*}\eta$,
and $J^{2}(\textrm{\ensuremath{\mathrm{\textrm{id}}}}_{\mathbf{R}^{n}}\times\zeta)^{*}(\pi^{2,1}){}^{*}\eta=G_{n}^{1}\zeta^{*}\eta$.
Hence 
\begin{equation}
(\pi^{2,1}){}^{*}\partial_{G_{n}^{1}\Xi}W_{h\rho}=i_{J^{2}\Xi}(E_{\rho}+F)+d((\pi^{2,1}){}^{*}i_{G_{n}^{1}\Xi}W_{h\rho}),\label{eq:FVF}
\end{equation}
and since the pull-back operation $J^{2}(\mathrm{\textrm{id}}_{\mathbf{R}^{n}}\times\zeta)^{*}$
annihilates contact forms, we get \eqref{eq:FirstVarFormula}.

(b) Formula \eqref{eq:IntegralVarFormula} follows from \eqref{eq:FirstVarFormula}
and application of the Stokes' theorem.
\end{proof}
An immersion $\zeta:\mathbf{R}^{n}\supset U\rightarrow Q$ is called
an \textit{\textcolor{black}{extremal}} of the variational functional
$\rho_{\Omega}$ \eqref{eq:Integral-W}, or \eqref{eq:Integral-rho},
\textit{\textcolor{black}{on piece}} $\Omega$, if for every vector
field $\Xi$ such that $\Xi$ vanishes on the boundary $\partial\Omega$
of $\Omega$ along $\zeta$, and $\textrm{supp}\,\Xi\cap\textrm{Im}\,\zeta\subset\zeta(\Omega)$,
the variational derivative of $\rho_{\Omega}$ satisfies
\[
(\partial_{G_{n}^{1}\Xi}W_{h\rho})_{\Omega}(\zeta)=0.
\]
This condition means that the values of the variational functional
$\rho_{\Omega}$ \textit{\textcolor{black}{remain stable}} with respect
to compact deformations of the \textit{\textcolor{black}{immersed
submanifold}} $\textrm{Im}\,\zeta$. The following theorem gives necessary
and sufficient conditions for $\zeta$ to be an extremal of the variational
functional $\rho_{\Omega}$ in terms of the Euler\textendash Lagrange
form $E_{\rho}$ \eqref{eq:Euler-Lagrange-form}.
\begin{thm}
\label{Thm:Euler-Lagrange}Let $\rho$ be a $\pi^{1,0}$-horizontal
Lepage form on $\mathrm{Imm}\,T_{n}^{1}Q$, and $\zeta:\mathbf{R}^{n}\supset U\rightarrow Q$
be an immersion. The following conditions are equivalent:

\textup{\textcolor{black}{(a)}} $\zeta$ is an extremal of $\rho_{\Omega}$.

\textup{\textcolor{black}{(b)}} The Euler\textendash Lagrange form
$E_{\rho}$ satisfies
\[
J^{2}(\mathrm{id}_{\mathbf{R}^{n}}\times\zeta)^{*}E_{\rho}=0.
\]

\textup{\textcolor{black}{(c)}} In any chart $(V,\psi)$ on $Q$,
the following system for $\zeta$ is satisfied
\begin{equation}
E_{K}(\mathscr{L})\circ T_{n}^{2}\zeta=0,\quad1\leq K\leq m+n,\label{eq:EL-equations}
\end{equation}
where $h\rho=\mathscr{L}\omega_{0}$, and the left-hand side of \eqref{eq:EL-equations}
is given by \eqref{eq:E-L-expressions}.
\end{thm}

\begin{proof}
We omit the proof since it follows from Theorem \ref{Thm:VarFormula},
(a), and it is a standard result of the~global variational theory
on fibered manifolds.
\end{proof}
The system of second-order partial differential equations \eqref{eq:EL-equations}
for $\zeta$ is called the \textit{\textcolor{black}{Euler\textendash Lagrange
equations}}, associated with the Lepage form $\rho$.

\subsection{Invariant Lepage forms on $G_{n}^{1}Q$ and the Noether's theorem}

Let $W$ be an open subset of $Q$, and $\tilde{W}^{1}=(\tau_{n,G}^{1,0})^{-1}(W)\subset G_{n}^{1}Q$.
Let $\eta$ be an $n$-form defined on $\tilde{W}^{1}$, and let $\alpha:W\rightarrow Q$
be a diffeomorphism. We say that $\eta$ is \textit{\textcolor{black}{invariant}}
with respect to $\alpha$, if 
\begin{equation}
(G_{n}^{1}\alpha)^{*}\eta=\eta\quad\textrm{mod}\,\tilde{\Theta}^{1}W,\label{eq:Invariance}
\end{equation}
where $G_{n}^{1}\alpha$ is the Grassmann prolongation of $\alpha$
(see \eqref{eq:GrassmannProlongationDiffeo}), and $\tilde{\Theta}^{1}W$
is the contact ideal (Sec. \ref{sec:3}). If condition \eqref{eq:Invariance}
is satisfied, $\alpha$ is also said to be an \textit{\textcolor{black}{invariance
transformation}} of $\eta$. It is straightforward to apply this definition
to vector fields by means of their one-parameter groups. We say that
a vector field $\Xi$ on $W$ is a \textit{\textcolor{black}{generator
of invariance transformations}} of the form $\eta$, if the one-parameter
group $\alpha_{t}^{\Xi}$ of $\Xi$ consists of invariance transformations
of $\eta$. 

Denote by $\alpha^{\Xi}$ the \textit{\textcolor{black}{global flow}}
of a vector field $\Xi$. We have the following basic formula applied
to the Grassmann prolongation of a vector field.
\begin{lem}
Let $\eta$ be an $n$-form on $\tilde{W}^{1}\subset G_{n}^{1}Q$,
and $\Xi$ be a vector field on $W$. Then
\begin{equation}
\frac{d}{dt}(G_{n}^{1}\alpha_{t}^{\Xi})^{*}\eta([J_{0}^{1}\zeta])=(G_{n}^{1}\alpha_{t}^{\Xi})^{*}\partial_{G_{n}^{1}\Xi}\eta([J_{0}^{1}\zeta]),\label{eq:LieDerivativeVectorField}
\end{equation}
for every point $(t,[J_{0}^{1}\zeta])$ of the domain of the flow
$G_{n}^{1}\alpha^{\Xi}$.
\end{lem}

\begin{proof}
Formula \eqref{eq:LieDerivativeVectorField} is obtained by a straightforward
differentiating of the curve $t\rightarrow(G_{n}^{1}\alpha_{t}^{\Xi})^{*}\eta([J_{0}^{1}\zeta])$
at an arbitrary point $(t,[J_{0}^{1}\zeta])$ of the domain of $G_{n}^{1}\alpha^{\Xi}$.
\end{proof}
The following result extends the classical \textit{\textcolor{black}{Noether's
equation}} to the Grassmann fibrations.
\begin{thm}
\label{thm:NoetherEq}Let $\eta$ be an $n$-form on $\tilde{W}^{1}\subset G_{n}^{1}Q$,
and $\Xi$ be a vector field on $W$. The following two conditions
are equivalent:

\textup{\textcolor{black}{(a)}} $\Xi$ is the generator of invariance
transformations of $\eta$.

\textup{\textcolor{black}{(b)}} The Lie derivative of $\eta$ with
respect to vector field $G_{n}^{1}\Xi$ satisfies
\[
\partial_{G_{n}^{1}\Xi}\eta=0\quad\mathrm{mod}\,\tilde{\Theta}^{1}W.
\]
\end{thm}

\begin{proof}
The proof is a straightforward extension of an analogous result; see
Urban and Krupka \cite{UK-IJGMMP}.
\end{proof}
\begin{cor}
Let $\rho$ be a $\pi^{1,0}$-horizontal Lepage form on $\mathrm{Imm}\,T_{n}^{1}Q$,
and $W_{h\rho}$ be the fundamental Lepage equivalent of the Lagrangian
$h\rho$. Let $\Xi$ be the generator of invariance transformations
of $W_{h\rho}$. Then
\begin{equation}
i_{J_{n}^{2}\Xi}E_{\rho}+d((\pi^{2,1}){}^{*}i_{G_{n}^{1}\Xi}W_{h\rho})=0\quad\mathrm{mod}\,\Theta^{2}W,\label{eq:FVF-invariance}
\end{equation}
where $\Theta^{2}W$ is the ideal of forms, locally generated by contact
$1$-forms $\omega^{K}=dy^{K}-y_{l}^{K}dx^{l}$, and $\omega_{j}^{K}=dy_{j}^{K}-y_{jl}^{K}dx^{l}$,
on a chart neighborhood of $J^{2}(\mathbf{R}^{n}\times Q)$.
\end{cor}

\begin{proof}
Using Theorem \ref{thm:NoetherEq}, formula \eqref{eq:FVF-invariance}
is a restatement of the first variation formula \eqref{eq:FVF} for
invariant Lepage forms.
\end{proof}
Let $\zeta:\mathbf{R}^{n}\supset U\rightarrow W\subset Q$ be an immersion.
A real-valued function $\phi:G_{n}^{1}Q\supset\tilde{W}^{1}\rightarrow\mathbf{R}$
is called a \textit{\textcolor{black}{level set function}} for the
immersion $\zeta$, if $d(\phi\circ G_{n}^{1}\zeta)=0$. Now we follow
the geometric ideas of the theory of higher-order variational functionals
on fibered manifolds (see Krupka \cite{Krupka-Lepage}), and we extend
the classical Noether's theorem to Grassmann fibrations.
\begin{thm}
\label{Thm:Noether}Let $\rho$ be a $\pi^{1,0}$-horizontal Lepage
form on $\mathrm{Imm}\,T_{n}^{1}Q$, and $W_{h\rho}$ be the fundamental
Lepage equivalent of the Lagrangian $h\rho$. Suppose an immersion
$\zeta:\mathbf{R}^{n}\supset U\rightarrow Q$ be an extremal of the
variational functional $\rho_{\Omega}$ \eqref{eq:Integral-W}. Then
for every generator $\Xi$ of invariance transformations of $W_{h\rho}$,
\begin{equation}
d(G_{n}^{1}\zeta^{*}i_{G_{n}^{1}\Xi}W_{h\rho})=0.\label{eq:NoetherCurrent}
\end{equation}
\end{thm}

\begin{proof}
\eqref{eq:NoetherCurrent} follows from the first variation formula
\eqref{eq:FVF-invariance} for invariant Lepage forms.
\end{proof}

\section{Example: Variational functional for minimal submanifolds \label{sec:6}}

Suppose $Q$ is a smooth manifold of dimension $n+m$, endowed with
a Riemannian metric $g$, and let $\zeta:X\rightarrow Q$ be an immersion
defined on an $n$-dimensional smooth manifold $X$. Let $(U,\varphi)$,
$\varphi=(x^{j})$, be a chart on $X$, and $(V,\psi)$, $\psi=(y^{K})$,
be a chart on $Q$ such that $\zeta(U)\subset V$. If $g$ has the
expression $g=g_{KL}dy^{K}\otimes dy^{L}$, then the \textit{\textcolor{black}{induced}}
Riemannian metric $\zeta^{*}g$ on $X$ is expressed by the formula
\[
\zeta^{*}g=(g_{KL}\circ\zeta)\frac{\partial(y^{K}\circ\zeta)}{\partial x^{j}}\frac{\partial(y^{L}\circ\zeta)}{\partial x^{k}}dx^{j}\otimes dx^{k},
\]
and the \textit{\textcolor{black}{induced}} Riemannian volume element
on $X$ reads
\[
\omega_{\zeta,g}=\sqrt{\det\left((g_{KL}\circ\zeta)\frac{\partial(y^{K}\circ\zeta)}{\partial x^{j}}\frac{\partial(y^{L}\circ\zeta)}{\partial x^{k}}\right)}\,dx^{1}\wedge dx^{2}\wedge\ldots\wedge dx^{n}.
\]

Consider now $X=\mathbf{R}^{n}$ with the global canonical coordinates
$x^{j}$, $1\leq j\leq n$. The Lagrangian for the problem of $n$-dimensional
\textit{\textcolor{black}{minimal submanifolds}} $\lambda=\mathscr{L}\omega_{0}$,
is defined by the positive-homogeneous function $\mathscr{L}:\textrm{Imm}\,T_{n}^{1}Q\rightarrow\mathbf{R}$,
\begin{equation}
\mathscr{L}(y^{K},y_{j}^{K})=\sqrt{\det(g_{KL}y_{j}^{K}y_{k}^{L})}.\label{eq:MinSubLagrange}
\end{equation}
Assigning to $\lambda$ the fundamental Lepage equivalent $W_{\lambda}$
\eqref{eq:FundLepEqZermelo}, we have the corresponding variational
functional (see \eqref{eq:Integral-W}) for immersions $\zeta$ of
$\mathbf{R}^{n}$ into $Q$, associated with a compact subset $\Omega$
of $\mathbf{R}^{n}$,
\begin{equation}
\zeta\rightarrow\rho_{\Omega}(\zeta)=\int_{\Omega}(G_{n}^{1}\zeta)^{*}W_{\lambda}.\label{eq:VarFuncMinSub}
\end{equation}

In \cite{Krupka-GlobalFunctionals}, Krupka proposed the $n$-form
$\omega$, given in a chart $(V,\psi)$ by
\begin{equation}
\omega_{\lambda}=\frac{1}{n!}\frac{1}{\mathscr{L}}g_{K_{1}L_{1}}g_{K_{2}L_{2}}\ldots g_{K_{n}L_{n}}D^{L_{1}L_{2}\ldots L_{n}}dy^{K_{1}}\land dy^{K_{2}}\wedge\ldots\wedge dy^{K_{n}},\label{eq:KrupkaForm}
\end{equation}
where $D^{L_{1}L_{2}\ldots L_{n}}=\det(y_{k}^{L_{j}})$, $1\leq j,k\leq n$,
which can be regarded as a Lagrangian for the problem of minimal submanifolds,
and satisfies the condition $T_{n}^{1}\zeta^{*}\omega_{\lambda}=\omega_{\zeta,g}$.
Note that $\omega_{\lambda}$ is a Lepage form on $\textrm{Imm}\,T_{n}^{1}Q$,
as introduced in Section \ref{sec:5-1}.
\begin{thm}
\label{Thm:MinSubLagran}Let $\lambda\in\Omega_{n,\mathbf{R}^{n}}^{1}(\mathbf{R}^{n}\times Q)$,
$\lambda=\mathscr{L}\omega_{0},$ be the Lagrangian of order $1$,
given by \eqref{eq:MinSubLagrange}. Then the Lepage equivalents of
$\lambda$, $W_{\lambda}$ \eqref{eq:FundLepEqZermelo}, $\Lambda_{\lambda}$
\eqref{eq:Hilbert-CaratheodoryForm}, and $\omega_{\lambda}$ \eqref{eq:KrupkaForm}
coincide, and are defined on the Grassmann fibration $G_{n}^{1}Q$.
\end{thm}

\begin{proof}
1. We show that $\Lambda_{\lambda}=\omega_{\lambda}$. Write
\begin{align}
\mathscr{L}^{2}=\det(g_{KL}y_{j}^{K}y_{k}^{L}) & =g_{Q_{1}L_{1}}g_{Q_{2}L_{2}}\ldots g_{Q_{n}L_{n}}y_{1}^{Q_{1}}y_{2}^{Q_{2}}\ldots y_{n}^{Q_{n}}D^{L_{1}L_{2}\ldots L_{n}},\label{eq:Lsquare}
\end{align}
and for any $p$, $1\leq p\leq n$, we get
\begin{align}
\frac{\partial\mathscr{L}}{\partial y_{p}^{K}} & =\frac{1}{\mathscr{L}}g_{Q_{1}L_{1}}\ldots g_{Q_{p-1}L_{p-1}}g_{KL_{p}}g_{Q_{p+1}L_{p+1}}\ldots g_{Q_{n}L_{n}}\label{eq:derivace}\\
 & \qquad\cdot y_{1}^{Q_{1}}\ldots y_{p-1}^{Q_{p-1}}y_{p+1}^{Q_{p+1}}\ldots y_{n}^{Q_{n}}D^{L_{1}L_{2}\ldots L_{n}}.\nonumber 
\end{align}
A straightforward application of \eqref{eq:derivace} now implies
the expression
\begin{align}
 & \frac{\partial\mathscr{L}}{\partial y_{1}^{K_{1}}}\frac{\partial\mathscr{L}}{\partial y_{2}^{K_{2}}}\ldots\frac{\partial\mathscr{L}}{\partial y_{n}^{K_{n}}}\nonumber \\
 & \quad\quad=\frac{1}{\mathscr{L}^{n}}g_{K_{1}L_{1}^{1}}g_{K_{2}L_{2}^{2}}\ldots g_{K_{n}L_{n}^{n}}D^{L_{1}^{1}L_{2}^{1}\ldots L_{n}^{1}}D^{L_{1}^{2}L_{2}^{2}\ldots L_{n}^{2}}\ldots D^{L_{1}^{n}L_{2}^{n}\ldots L_{n}^{n}}\label{eq:product-derivatives}\\
 & \quad\quad\cdot g_{Q_{2}^{1}L_{2}^{1}}\ldots g_{Q_{n}^{1}L_{n}^{1}}g_{Q_{1}^{2}L_{1}^{2}}g_{Q_{3}^{2}L_{3}^{2}}\ldots g_{Q_{n}^{2}L_{n}^{2}}\ldots\ldots g_{Q_{1}^{n}L_{1}^{n}}g_{Q_{2}^{n}L_{2}^{n}}\ldots g_{Q_{n-1}^{n}L_{n-1}^{n}}\nonumber \\
 & \quad\quad\cdot y_{2}^{Q_{2}^{1}}\ldots y_{n}^{Q_{n}^{1}}y_{1}^{Q_{1}^{2}}y_{3}^{Q_{3}^{2}}\ldots y_{n}^{Q_{n}^{2}}\ldots\ldots y_{1}^{Q_{1}^{n}}y_{2}^{Q_{2}^{n}}\ldots y_{n-1}^{Q_{n-1}^{n}}.\nonumber 
\end{align}
However, from the symmetry and skew-symmetry properties of the expression
\eqref{eq:product-derivatives}, we are allowed to replace the term
$D^{L_{1}^{1}L_{2}^{1}\ldots L_{n}^{1}}D^{L_{1}^{2}L_{2}^{2}\ldots L_{n}^{2}}\ldots D^{L_{1}^{n}L_{2}^{n}\ldots L_{n}^{n}}$
by the product of determinants
\begin{equation}
\frac{1}{n!}\prod_{\tau}D^{L_{1}^{\tau(1)}L_{2}^{\tau(2)}\ldots L_{n}^{\tau(n)}},\label{eq:Determinants}
\end{equation}
where $\tau$ runs through the cyclic permutations $(1,2,\ldots,n)$,
$(n,1,2,\ldots,n-1)$, $\ldots$, $(2,3,\ldots,n,1)$. Using \eqref{eq:Lsquare},
\eqref{eq:Determinants}, we now get
\begin{align*}
 & \Lambda_{\lambda}=\frac{1}{\mathscr{L}^{n-1}}\frac{\partial\mathscr{L}}{\partial y_{1}^{K_{1}}}\frac{\partial\mathscr{L}}{\partial y_{2}^{K_{2}}}\ldots\frac{\partial\mathscr{L}}{\partial y_{n}^{K_{n}}}dy^{K_{1}}\land dy^{K_{2}}\wedge\ldots\wedge dy^{K_{n}}\\
 & \quad\quad=\frac{1}{n}\frac{1}{\mathscr{L}^{2n-1}}g_{K_{1}L_{1}^{1}}g_{K_{2}L_{2}^{2}}\ldots g_{K_{n}L_{n}^{n}}D^{L_{1}^{1}L_{2}^{2}\ldots L_{n}^{n}}\\
 & \quad\quad\cdot\left(g_{Q_{1}^{n}L_{1}^{n}}g_{Q_{2}^{1}L_{2}^{1}}g_{Q_{3}^{2}L_{3}^{2}}\ldots g_{Q_{n}^{n-1}L_{n}^{n-1}}y_{1}^{Q_{1}^{n}}y_{2}^{Q_{2}^{1}}\ldots y_{n}^{Q_{n}^{n-1}}D^{L_{1}^{n}L_{2}^{1}L_{3}^{2}\ldots L_{n}^{n-1}}\right)\\
 & \quad\quad\ldots\cdot\left(g_{Q_{1}^{2}L_{1}^{2}}g_{Q_{2}^{3}L_{2}^{3}}\ldots g_{Q_{n-1}^{n}L_{n-1}^{n}}g_{Q_{n}^{1}L_{n}^{1}}y_{1}^{Q_{1}^{n}}y_{2}^{Q_{2}^{1}}\ldots y_{n}^{Q_{n}^{n-1}}D^{L_{1}^{2}L_{2}^{3}\ldots L_{n-1}^{n}L_{n}^{1}}\right)\\
 & \quad\quad\quad dy^{K_{1}}\land dy^{K_{2}}\wedge\ldots\wedge dy^{K_{n}}\\
 & \quad\quad=\frac{1}{n!}\frac{(\mathscr{L}^{2})^{n-1}}{\mathscr{L}^{2n-1}}g_{K_{1}L_{1}^{1}}g_{K_{2}L_{2}^{2}}\ldots g_{K_{n}L_{n}^{n}}D^{L_{1}^{1}L_{2}^{2}\ldots L_{n}^{n}}dy^{K_{1}}\land dy^{K_{2}}\wedge\ldots\wedge dy^{K_{n}}\\
 & \quad\quad=\omega_{\lambda}.
\end{align*}

2. To show that $W_{\lambda}=\omega_{\lambda}$, we proceed by a straightforward
differentiating of $\mathscr{L}$ with respect to the variables $y_{j}^{K_{j}}$,
$1\leq j\leq n$, with the help of \eqref{eq:derivace}, \eqref{eq:Determinants},
and the formula
\begin{align*}
 & \frac{\partial D^{L_{1}L_{2}\ldots L_{n}}}{\partial y_{p}^{M}}=\delta_{1}^{p}\left|\begin{array}{ccc}
\delta_{M}^{L_{1}} & \delta_{M}^{L_{2}} & \ldots\delta_{M}^{L_{n}}\\
y_{2}^{L_{1}} & y_{2}^{L_{2}} & \ldots y_{2}^{L_{n}}\\
\vdots & \vdots & \vdots\\
y_{n}^{L_{1}} & y_{n}^{L_{2}} & \ldots y_{n}^{L_{n}}
\end{array}\right|+\ldots+\delta_{n}^{p}\left|\begin{array}{ccc}
y_{1}^{L_{1}} & y_{1}^{L_{2}} & \ldots y_{1}^{L_{n}}\\
y_{2}^{L_{1}} & y_{2}^{L_{2}} & \ldots y_{2}^{L_{n}}\\
\vdots & \vdots & \vdots\\
\delta_{M}^{L_{1}} & \delta_{M}^{L_{2}} & \ldots\delta_{M}^{L_{n}}
\end{array}\right|.
\end{align*}
Hence we obtain

\begin{align*}
\frac{\partial^{n}\mathscr{L}}{\partial y_{1}^{K_{1}}\partial y_{2}^{K_{2}}\ldots\partial y_{n}^{K_{n}}} & =\frac{1}{\mathscr{L}}g_{K_{1}L_{1}}g_{K_{2}L_{2}\ldots}g_{K_{n}L_{n}}D^{L_{1}L_{2}\ldots L_{n}},
\end{align*}
as required.
\end{proof}
\begin{rem}
A relationship between the Hilbert-Carath\'eodory form $\Lambda_{\lambda}$
and the fundamental Lepage equivalent $W_{\lambda}$ of the corresponding
positive-homogeneous Lagrangian $\lambda$ was also studied by Crampin
and Saunders \cite{CramSaun3} for \textit{trivial} Lagrangians; it
is proved that a Lagrange function of the form $\mathscr{L}=\det(d_{j}f^{k})$
, where $f^{k}$, $1\leq k\leq n$, are some functions on $Q$, has
the property that the forms $\Lambda_{\lambda}$ and $W_{\lambda}$
coincide. Theorem \ref{Thm:MinSubLagran} shows, however, a \textit{different}
example of a positive-homogeneous Lagrangian with this property, the
minimal submanifolds Lagrangian.
\end{rem}

By Theorem \ref{Thm:Euler-Lagrange}, \eqref{eq:EL-equations}, the
equations for extremals, associated with the Lagrange function $\mathscr{L}$
\eqref{eq:MinSubLagrange}, read 
\begin{equation}
E_{K}(\mathscr{L})\circ T_{n}^{2}\zeta=0,\quad1\leq K\leq m+n,\label{eq:MinSubmanifoldEquation}
\end{equation}
where
\begin{align*}
E_{K}(\mathscr{L}) & =\frac{n}{2}\frac{1}{\mathscr{L}}\frac{\partial g_{Q_{1}L_{1}}}{\partial y^{K}}g_{Q_{2}L_{2}}\ldots g_{Q_{n}L_{n}}y_{1}^{Q_{1}}y_{2}^{Q_{2}}\ldots y_{n}^{Q_{n}}D^{L_{1}L_{2}\ldots L_{n}}\\
 & -\sum_{j=1}^{n}d_{j}\Bigl(\frac{1}{\mathscr{L}}g_{Q_{1}L_{1}}\ldots g_{Q_{j-1}L_{j-1}}g_{KL_{j}}g_{Q_{j+1}L_{j+1}}\ldots g_{Q_{n}L_{n}}\\
 & \quad\quad\quad\quad\cdot y_{1}^{Q_{1}}\ldots y_{j-1}^{Q_{j-1}}y_{j+1}^{Q_{j+1}}\ldots y_{n}^{Q_{n}}D^{L_{1}L_{2}\ldots L_{n}}\Bigr).
\end{align*}
\eqref{eq:MinSubmanifoldEquation} is the \textit{minimal submanifolds
equation}.

We conclude this section with an application of the Noether's theorem
(Theorem \ref{Thm:Noether}) to the problem of $2$-dimensional immersed
minimal submanifolds of a Euclidean space.

Consider the Euclidean space $Q=\mathbf{R}^{m+2}$ with its canonical
smooth manifold structure. Let $\Xi$ be a vector field on $\mathbf{R}^{m+2}$,
and $G_{2}^{1}\Xi$ be its first-order Grassmann prolongation, expressed
in the $(i)=(i_{1},i_{2})$-subordinate chart $(\tilde{V}_{n}^{1(i)},\tilde{\chi}_{n}^{1(i)})$,
$\tilde{\chi}_{n}^{1(i)}=(w^{i_{1}},w^{i_{2}},w^{\mu},w_{i_{1}}^{\mu},w_{i_{2}}^{\mu})$,
on $G_{2}^{1}\mathbf{R}^{m+2}$ by
\begin{align*}
\Xi & =\Xi^{i_{1}}\frac{\partial}{\partial w^{i_{1}}}+\Xi^{i_{2}}\frac{\partial}{\partial w^{i_{2}}}+\Xi^{\mu}\frac{\partial}{\partial w^{\mu}},\\
G_{2}^{1}\Xi & =\Xi^{i_{1}}\frac{\partial}{\partial w^{i_{1}}}+\Xi^{i_{2}}\frac{\partial}{\partial w^{i_{2}}}+\Xi^{\mu}\frac{\partial}{\partial w^{\mu}}+\Xi_{i_{1}}^{\mu}\frac{\partial}{\partial w_{i_{1}}^{\mu}}+\Xi_{i_{2}}^{\mu}\frac{\partial}{\partial w_{i_{2}}^{\mu}},
\end{align*}
where $\Xi_{p}^{\mu}$, $p=i_{1},i_{2}$, has the expression (cf.
Lemma \eqref{Lem:VectorFieldProlong})
\begin{align*}
 & \Xi_{p}^{\mu}=\frac{\partial\Xi^{\mu}}{\partial w^{p}}+w_{p}^{\nu}\frac{\partial\Xi^{\mu}}{\partial w^{\nu}}-w_{i_{1}}^{\mu}\frac{\partial\Xi^{i_{1}}}{\partial w^{p}}-w_{i_{2}}^{\mu}\frac{\partial\Xi^{i_{2}}}{\partial w^{p}}-w_{i_{1}}^{\mu}w_{p}^{\nu}\frac{\partial\Xi^{i_{1}}}{\partial w^{\nu}}-w_{i_{2}}^{\mu}w_{p}^{\nu}\frac{\partial\Xi^{i_{2}}}{\partial w^{\nu}}
\end{align*}
(no summation through $i_{1}$, $i_{2}$). A straightforward computation
of the Lie derivative of the Lepage equivalent $W_{\lambda}=\omega_{\lambda}$
with respect to $G_{2}^{1}\Xi$ leads to the \textit{Noether's equation}
$\partial_{G_{2}^{1}\Xi}\omega_{\lambda}=0$ $\mathrm{mod}\,\tilde{\Theta}^{1}$
for unknowns $\Xi$, which has the form
\begin{equation}
\partial_{G_{2}^{1}\Xi}\tilde{\mathscr{L}_{G}}+\tilde{\mathscr{L}_{G}}\left(\frac{\partial\Xi^{i_{1}}}{\partial w^{i_{1}}}+\frac{\partial\Xi^{i_{1}}}{\partial w^{\nu}}w_{i_{1}}^{\nu}+\frac{\partial\Xi^{i_{2}}}{\partial w^{i_{2}}}+\frac{\partial\Xi^{i_{2}}}{\partial w^{\nu}}w_{i_{2}}^{\nu}\right)=0,\label{eq:NoetherEqMinSub}
\end{equation}
where $\partial_{G_{2}^{1}\Xi}\tilde{\mathscr{L}_{G}}$ denotes the
Lie derivative of function $\tilde{\mathscr{L}_{G}}:G_{2}^{1}\mathbf{R}^{m+2}\rightarrow\mathbf{R}$
with respect to $G_{2}^{1}\Xi$, and $\tilde{\mathscr{L}_{G}}$ is
uniquely defined Grassmann projection of the Lagrange function $\mathscr{L}$
\eqref{eq:MinSubLagrange}, $\mathscr{L}=(w_{1}^{i_{1}}w_{2}^{i_{2}}-w_{1}^{i_{2}}w_{2}^{i_{1}})\cdot\tilde{\mathscr{L}_{G}}$
(see \eqref{eq:G-function}). 

Moreover, if $Q=\mathbf{R}^{m+2}$ is considered with the \textit{\textcolor{black}{Euclidean
metric}} $g=(\delta_{KL})$, it is easy to show that a vector field
$\xi$ with \textit{\textcolor{black}{constant}} \textit{\textcolor{black}{coefficients}}
is a general solution of equation \eqref{eq:NoetherEqMinSub}. Hence,
we obtain $(m+2)$-dimensional Lie algebra of generators of invariance
transformations of the Lepage equivalent $\omega_{\lambda}$, generated
by the vector fields $\xi_{1}=\partial/\partial w^{i_{1}}$, $\xi_{2}=\partial/\partial w^{i_{2}}$,
$\xi_{\mu}=\partial/\partial w^{\mu}$. Contracting the Lepage form
$\omega_{\lambda}$ by the generators of invariance transformations,
we obtain the \textit{\textcolor{black}{Noether currents}} (first
integrals), locally expressed by
\begin{align}
i_{\xi_{1}}\omega_{\lambda} & =\frac{1}{\tilde{\mathscr{L}_{G}}}\left(\sum_{\mu}w_{i_{2}}^{\mu}dw^{\mu}+dw^{i_{2}}\right),\quad i_{\xi_{2}}\omega_{\lambda}=-\frac{1}{\tilde{\mathscr{L}_{G}}}\left(\sum_{\mu}w_{i_{1}}^{\mu}dw^{\mu}+dw^{i_{1}}\right),\nonumber \\
i_{\xi_{\mu}}\omega_{\lambda} & =\frac{1}{\tilde{\mathscr{L}_{G}}}\left(\sum_{\sigma}\left(w_{i_{1}}^{\mu}w_{i_{2}}^{\sigma}-w_{i_{1}}^{\sigma}w_{i_{2}}^{\mu}\right)dw^{\sigma}-w_{i_{2}}^{\mu}dw^{i_{1}}+w_{i_{1}}^{\mu}dw^{i_{2}}\right),\label{eq:NoetherCurrents}
\end{align}
where
\[
\tilde{\mathscr{L}_{G}}=\sqrt{\sum_{\sigma_{1}<\sigma_{2}}\left(w_{i_{1}}^{\sigma_{1}}w_{i_{2}}^{\sigma_{2}}-w_{i_{1}}^{\sigma_{2}}w_{i_{2}}^{\sigma_{1}}\right)^{2}+\sum_{\sigma}\left((w_{i_{1}}^{\sigma})^{2}+(w_{i_{2}}^{\sigma})^{2}\right)+1}.
\]
By Theorem \ref{Thm:Noether}, linear forms \eqref{eq:NoetherCurrents}
are \textit{\textcolor{black}{closed}} hence also \textit{\textcolor{black}{exact}}
\textit{\textcolor{black}{along}} every extremal $\zeta:\mathbf{R}^{2}\supset U\rightarrow\mathbf{R}^{m+2}$
of the variational functional \eqref{eq:VarFuncMinSub}. This result
extends the concept of a Noether current as a level-set function for
extremals, known in the classical mechanics (cf. Urban and Krupka
\cite{UK-IJGMMP,UK-AMAPN}).

Finally, we consider 2-dimensional nonparametric minimal surfaces
of $\mathbf{R}^{3}$. Write $y^{1}=x$, $y^{2}=y$, $y^{3}=z$, the
canonical global coordinates of $\mathbf{R}^{3}$, and $y_{j}^{1}=x_{j}$,
$y_{j}^{2}=y_{j}$, $y_{j}^{3}=z_{j}$, $j=1,2$, the associated coordinates
of $T_{2}^{1}\mathbf{R}^{3}$. Consider a subordinate chart $(\tilde{V}_{2}^{1,(i)},\tilde{\chi}_{2}^{1,(i)})$
on $G_{2}^{1}\mathbf{R}^{3}=\mathrm{Imm}\,T_{2}^{1}\mathbf{R}^{3}/GL_{2}(\mathbf{R})$,
where, for instance, $(i_{1},i_{2})=(1,2)$, and $\tilde{\chi}_{2}^{1,(i)}=(w^{1},w^{2},w^{3},w_{1}^{3},w_{2}^{3})$,
where $w^{1}=x$, $w^{2}=y$, $w^{3}=z$, and 
\[
w_{1}^{3}=\frac{y_{2}z_{1}-y_{1}z_{2}}{x_{1}y_{2}-x_{2}y_{1}},\quad w_{2}^{3}=\frac{x_{1}z_{2}-x_{2}z_{1}}{x_{1}y_{2}-x_{2}y_{1}}
\]
(cf. Section \ref{sec:3}, \eqref{eq:AdaptCoord}). It is now easy
to verify that for $Q=\mathbf{R}^{3}$ with the \textcolor{black}{Euclidean
metric} $\delta_{KL}$, the Euler\textendash Lagrange equations \eqref{eq:MinSubmanifoldEquation}
for graph of a surface $\zeta:\mathbf{R}^{2}\supset U\rightarrow\mathbf{R}^{3}$,
$\zeta(x,y)=(x,y,u(x,y))$, are equivalent to the well-known single
second-order differential equation for an unknown function $u:\mathbf{R}^{2}\supset U\rightarrow\mathbf{R}$,
\begin{equation}
(1+u_{y}^{2})u_{xx}-2u_{x}u_{y}u_{xy}+(1+u_{x}^{2})u_{yy}=0,\label{eq:LagrangeMinSub}
\end{equation}
cf. Dierkes, Hildebrandt, and Sauvigny \cite{Dierkes}. 

Contracting the Lepage form $\omega_{\lambda}$ by the generators
of invariance transformations $\partial/\partial x$, $\partial/\partial y$,
$\partial/\partial z$, we get the Noether currents \eqref{eq:NoetherCurrents}
hence the \textit{conservation law equations} of the form
\begin{align}
\frac{(x_{1}y_{2}-x_{2}y_{1})dy-(z_{1}x_{2}-z_{2}x_{1})dz}{\sqrt{\left(y_{1}z_{2}-y_{2}z_{1}\right)^{2}+\left(z_{1}x_{2}-z_{2}x_{1}\right)^{2}+\left(x_{1}y_{2}-x_{2}y_{1}\right)^{2}}} & =df(x,y),\nonumber \\
\frac{(y_{1}z_{2}-y_{2}z_{1})dz-(x_{1}y_{2}-x_{2}y_{1})dx}{\sqrt{\left(y_{1}z_{2}-y_{2}z_{1}\right)^{2}+\left(z_{1}x_{2}-z_{2}x_{1}\right)^{2}+\left(x_{1}y_{2}-x_{2}y_{1}\right)^{2}}} & =dg(x,y),\label{eq:ConservationLawEquation}\\
\frac{(z_{1}x_{2}-z_{2}x_{1})dx-(y_{1}z_{2}-y_{2}z_{1})dy}{\sqrt{\left(y_{1}z_{2}-y_{2}z_{1}\right)^{2}+\left(z_{1}x_{2}-z_{2}x_{1}\right)^{2}+\left(x_{1}y_{2}-x_{2}y_{1}\right)^{2}}} & =dh(x,y),\nonumber 
\end{align}
where $f$, $g$, $h$, are arbitrary functions on $U\subset\mathbf{R}^{2}$.
The Noether theorem \ref{Thm:Noether} says that every minimal surface
$\zeta$ of $\mathbf{R}^{3}$ is a solution of \eqref{eq:ConservationLawEquation}.
Conversely, we claim that every solution $\zeta:\mathbf{R}^{2}\supset U\rightarrow\mathbf{R}^{3}$,
$\zeta(x,y)=(x,y,u(x,y))$, of the conservation law equations \eqref{eq:ConservationLawEquation}
is a minimal surface of $\mathbf{R}^{3}$ hence an extremal of the
variational functional \eqref{eq:VarFuncMinSub}. Indeed, using the
coordinate expressions of $G_{2}^{1}\zeta:\mathbf{R}^{2}\supset U\rightarrow G_{2}^{1}\mathbf{R}^{3}$
\eqref{eq:GrassmannProlong}, equations \eqref{eq:ConservationLawEquation}
form a Pfaffian system
\begin{align}
 & \frac{u_{x}u_{y}dx+\left(1+\left(u_{y}\right)^{2}\right)dy}{\sqrt{\left(u_{x}\right)^{2}+\left(u_{y}\right)^{2}+1}}=df,\quad\frac{-\left(1+\left(u_{x}\right)^{2}\right)dx-u_{x}u_{y}dy}{\sqrt{\left(u_{x}\right)^{2}+\left(u_{y}\right)^{2}+1}}=dg,\label{eq:Vlastnosti-f-g}
\end{align}
and
\[
\frac{-u_{y}dx+u_{x}dy}{\sqrt{\left(u_{x}\right)^{2}+\left(u_{y}\right)^{2}+1}}=dh.
\]
These conditions imply that $u_{x}df(x,y)+u_{y}dg(x,y)=dh(x,y)$ or,
equivalently,
\begin{equation}
u_{x}\frac{\partial f}{\partial x}+u_{y}\frac{\partial g}{\partial x}-\frac{\partial h}{\partial x}=0,\quad u_{x}\frac{\partial f}{\partial y}+u_{y}\frac{\partial g}{\partial y}-\frac{\partial h}{\partial y}=0.\label{eq:Rovnice-ufg}
\end{equation}
Differentiating the first equation of \eqref{eq:Rovnice-ufg} with
respect to $y$, the latter one with respect to $x$, and subtracting
we get
\[
u_{xx}\frac{\partial f}{\partial y}-u_{xy}\left(\frac{\partial f}{\partial x}-\frac{\partial g}{\partial y}\right)-u_{yy}\frac{\partial g}{\partial x}=0.
\]
Substituting now into this equation the partial derivatives of $f$
and $g$, determined by \eqref{eq:Vlastnosti-f-g}, we conclude that
$u=u(x,y)$ is a solution of the minimal surface equation \eqref{eq:LagrangeMinSub}.
\begin{rem}
Equivalence of the Euler\textendash Lagrange equations for extremals
on one side, and a system of conservation law equations on the other
side, is \textit{\textcolor{black}{not}} understood in general yet.
Our results for the minimal submanifolds problem extend particular
examples from geometric mechanics which illustrate this phenomena,
see Urban and Krupka \cite{UK-IJGMMP,UK-AMAPN}. 
\end{rem}


\begin{thebibliography}{10}
\bibitem{Betounes}D.E. Betounes, Extension of the classical Cartan
form, \textit{Phys. Rev. D} 29 (1984), 599\textendash 606.

\bibitem{BrajercikKrupka}J. Brajer\v{c}\'{i}k and D. Krupka, Variational
principles for locally variational forms, \textit{J. Math. Phys.}
46 (052903) (2005), 1\textendash 15.

\bibitem{Caratheodory}C. Carathéodory, \"Uber die Variationsrechnung
bei mehrfachen Integralen, \textit{Acta Szeged Sect. Sci. Math.} 4
(1929), 193\textendash 216.

\bibitem{CramSaun1}M. Crampin, D.J. Saunders, The Hilbert-Carath\'{e}odory
form for parametric multiple integral problems in the calculus of
variations, \textit{Acta Appl. Math.} 76 (2003), 37\textendash 55.

\bibitem{CramSaun3}M. Crampin and D. J. Saunders, On null Lagrangians,
\textit{Diff. Geom. Appl.} 22 (2005), 131\textendash 146.

\bibitem{Dedecker}P. Dedecker, On the generalization of symplectic
geometry to multiple integrals in the calculus of variations, in:
K. Bleuler and A. Reetz (eds), \textit{\textcolor{black}{Differential
Geometrical Methods in Mathematical Physics}}, Lecture Notes in Mathematics,
Vol. 570 (Springer, Berlin, 1977), pp. 395\textendash 456.

\bibitem{Dierkes}U. Dierkes, S. Hildebrandt, and F. Sauvigny, \textit{Minimal
surfaces}, 2nd Edition, Grundlehren math. Wissenschaften 339, Springer-Verlag,
Berlin, 2010.

\bibitem{GMS-FieldTheory}G. Giachetta, L. Mangiarotti, and G. Sardanashvily,
\textit{\textcolor{black}{Advanced classical field theory}}, World
Scientific, Singapore, 2009.

\bibitem{Grigore-Handbook}D. R. Grigore, Lagrangian formalism on
Grassmann manifolds, in: D. Krupka and D. Saunders (Eds.), \textit{Handbook
of Global Analysis}, Elsevier, Amsterdam, 2008, pp. 325\textendash 371.

\bibitem{Grigore}D. R. Grigore and D. Krupka, Invariants of velocities
and higher-order Grassmann bundles, \textit{J. Geom. Phys.} 24 (1998),
244\textendash 264.

\bibitem{Kossmann}Y. Kossmann-Schwarzbach, \textit{\textcolor{black}{The
Noether Theorems}}, Springer-Verlag, New York, 2011.

\bibitem{Krupka-Fund.Lep.eq.}D. Krupka, A map associated to the Lepagean
forms of the calculus of variations in fibered manifolds, \textit{Czech.
Math. J.} 27 (1977), 114\textendash 118.

\bibitem{Krupka-Lepage}D. Krupka, Lepagean forms in higher-order
variational theory, in: \textit{\textcolor{black}{Modern Developments
in Analytical Mechanics}} (Academy of Sciences of Turin, Elsevier,
1983), pp. 197\textendash 238.

\bibitem{Krupka-GlobalFunctionals}D. Krupka, Global variational functionals
on fibered spaces, \textit{Nonlinear Analysis} 47 (2001) 2633\textendash 2642. 

\bibitem{Krupka-Debrecen}D. Krupka, Lepage forms in Kawaguchi spaces
and the Hilbert form, \textit{Publ. Math. Debrecen} 84, 1-2 (2014)
147\textendash 164.

\bibitem{Krupka-Book}D. Krupka, \textit{Introduction to Global Variational
Geometry}, Atlantis Studies in Variational Geometry, Vol. 1, Atlantis
Press, Amsterdam\textendash Beijing\textendash Paris, 2015.

\bibitem{KKS}D. Krupka, O. Krupkov\'a, and D. Saunders, The Cartan
form and its generalizations in the calculus of variations, Int. J.
Geom. Meth. Mod. Phys. 7(4) (2010) 631\textendash 654.

\bibitem{MannoVitolo}G. Manno and R. Vitolo, Geometric aspects of
higher-order variational principles on submanifolds, \textit{Acta
Appl. Math.} 101 (2008) 215\textendash 229.

\bibitem{McKiernan}M. A. McKiernan, Sufficiency of parameter invariance
conditions in areal and higher-order Kawaguchi spaces, \textit{Publ.
Math. Debrecen} 13 (1966) 77\textendash 85.

\bibitem{Saunders-JGP} D.J. Saunders and M. Crampin, The fundamental
form of a homogeneous Lagrangian in two independent variables, \textit{J.
Geom. Phys.} 60, No. 11 (2010), 1681\textendash 1697.

\bibitem{UK-Acta}Z. Urban and D. Krupka, Variational sequences in
mechanics on Grassmann fibrations, \textit{Acta Appl. Math.} 112(2)
(2010) 225\textendash 249.

\bibitem{UK-IJGMMP}Z. Urban and D. Krupka, Foundations of higher-order
variational theory on Grassmann fibrations, \textit{Int. J. Geom.
Meth. Mod. Phys.} 11, No. 7 (2014) 1460023.

\bibitem{UK-AMAPN}Z. Urban and D. Krupka, Variational theory on Grassmann
fibrations: Examples, \textit{Acta Math. Acad. Paed. Nyíregyhasiensis}
31, No. 1 (2015) 153\textendash 170.

\bibitem{UK-Debrecen}Z. Urban and D. Krupka, The Zermelo conditions
and higher order homogeneous functions, \textit{Publ. Math. Debrecen}
82, No. 1 (2013) 59\textendash 76.

\bibitem{Volna}J. Volná and Z. Urban, First-order Variational Sequences
in Field Theory, in: D. Zenkov (Ed.), \textit{\textcolor{black}{The
Inverse Problem of the Calculus of Variations}}, \textit{\textcolor{black}{Local
and Global Theory}}, Atlantis Press, Amsterdam\textendash Beijing\textendash Paris,
2015, pp. 215\textendash 284.
\end{thebibliography}
\end{document}